\documentclass{amsart}
\usepackage{amsmath,amssymb}
\usepackage{color,mathdots}

\newtheorem{theorem}{Theorem}[section]
\newtheorem{lemma}[theorem]{Lemma}
\newtheorem{corollary}[theorem]{Corollary}
\newtheorem{fact}[theorem]{Fact}
\newtheorem{proposition}[theorem]{Proposition}

\theoremstyle{definition}

\newtheorem{remark}[theorem]{Remark}
\newtheorem{definition}[theorem]{Definition}

\def \d {\delta}

\def \Ld {{\mathcal L}_\d}
\def \La {{\mathcal L}_{\d,\bar a}}
\def \Ll {{\mathcal L}_{\d,\ell}}
\def \SCF{\operatorname{SCF}}
\def \SCFe{\operatorname{SCF}_{p,e}}
\def \SCFl{\operatorname{SCF}_{p,e}^{\lambda}}
\def \DF{\operatorname{DF}}
\def \DCF{\operatorname{DCF}}
\def \SDCF{\operatorname{SDCF}}
\def \SDCFe{\operatorname{SDCF}_{p,\epsilon}}
\def \SDCFa{\operatorname{SDCF}_{p,\epsilon}^{\bar a}}
\def \SDCFl{\operatorname{SDCF}_{p,\epsilon}^{\ell}}

\def \tp{\operatorname{\tp}}

\def \ord {\operatorname{ord}}
\def \rk {\operatorname{rank}}
\def \NN {{\mathbb N_0}}

\def\Ind#1#2{#1\setbox0=\hbox{$#1x$}\kern\wd0\hbox to 0pt{\hss$#1\mid$\hss}
\lower.9\ht0\hbox to 0pt{\hss$#1\smile$\hss}\kern\wd0}

\def\Notind#1#2{#1\setbox0=\hbox{$#1x$}\kern\wd0\hbox to 0pt{\mathchardef
\nn=12854\hss$#1\nn$\kern1.4\wd0\hss}\hbox to
0pt{\hss$#1\mid$\hss}\lower.9\ht0 \hbox to
0pt{\hss$#1\smile$\hss}\kern\wd0}

\title[Separably differentially closed fields]{Separably differentially closed fields}

\author{Kai Ino}
\address{Kai Ino, Department of Mathematics, University of Manchester, Oxford Road, Manchester, United Kingdom M13 9PL}
\email{kai.ino@manchester.ac.uk}

\author{Omar Le\'on S\'anchez}
\address{Omar Le\'on S\'anchez, Department of Mathematics, University of Manchester, Oxford Road, Manchester, United Kingdom M13 9PL}
\email{omar.sanchez@manchester.ac.uk}

\date{\today}
\thanks{{\em Acknowledgements}: The second author was partially supported by EPSRC grant EP/V03619X/1}

\begin{document}

\maketitle

\begin{abstract}
We introduce and study a new class of differential fields in positive characteristic. We call them separably differentially closed fields and demonstrate that they are the differential analogue of separably closed fields. We prove several (algebraic and model-theoretic) properties of this class. Among other things, we show that it is an elementary class, whose theory we denote $\SDCF$, and that its completions are determined by specifying the characteristic $p$ and the differential degree of imperfection $\epsilon$. Furthermore, after adding what we call the differential $\lambda$-functions, we prove that the theory $\SDCFl$ admits quantifier elimination, is stable, and prime model extensions exist.
\end{abstract}

\tableofcontents

\section{Introduction}

The theory of separably closed fields SCF, from a model-theoretic perspective, has been a subject of study since the late 1960s; see for instance \cite{Delon88,Ershov67,Srour86,Wood79}. Ershov's famous paper \cite{Ershov67} shows that the completion of SCF are determined by specifying the characteristic $p$ and, in case $p>0$, the degree of imperfection $e$. Recall that, for a field $K$, the degree of imperfection (aka Ershov invariant) is determined by $[K:K^p]=p^e$, where $e=\infty$ when the extension $K/K^p$ is of infinite degree. 

\medskip

Later on, around the 1980s, several papers appeared developing further some of the model-theoretic properties of the theory $\SCFe$. We recall some of these. In~\cite{Wood79}, one finds proof that $\SCFe$ is a stable theory. After adding the so-called $\lambda$-functions, we denote the theory by $\SCFl$. In \cite{Delon88}, it is shown that $\SCFl$ admits quantifier elimination and, when $e$ is finite, it also admits elimination of imaginaries. Also, it is known that $\SCFl$ is an equational theory with a natural algebraic description of forking independence \cite{MZ20,Srour86}, and furthermore, for finite $e$, it satisfies the non-finite cover property \cite[\S 4.4]{MMP}. 

\medskip

On the other hand, the theory of differentially closed fields in characteristic $p>0$, DCF$_p$, was initiated in the 1970s by the work of Wood in \cite{Wood73} followed by \cite{Wood74,Wood76}. It was shown there that DCF$_p$ is complete, model-complete, stable, and after adding the $p$-th root function on constants it admits quantifier elimination. In many ways, DCF$_p$ is the differential analogue of the theory ACF$_p$ (just as the better known DCF$_0$ is the differential analogue of ACF$_0$). One natural question to ask is: what theory of differential fields in characteristic $p$ is a suitable differential analogue of $\SCF_p$? To the authors' knowledge, this question has not been addressed elsewhere and was in fact the main motivation for this paper.

\medskip

In this paper we introduce a new class of differential fields that we call \emph{separably differentially closed} and, as witnessed by our results, demonstrate that they are the differential analogue of separably closed fields. The definition takes its cue from the fact that a field is separably closed if and only if it is existentially closed (in the language of fields) in every separable extension. 

\begin{definition}
A differential field $(K,\d)$, of arbitrary characteristic, is said to be separably differentially closed if it is existentially closed (in the language of differential fields) in every differential field extension that is a separable extension (as fields). In other words, if $(L,\d)$ is a differential extension of $(K,\d)$ and $L/K$ is separable, then $(K,\d)$ is e.c. in $(L,\d)$.
\end{definition}

In characteristic zero being separably differentially closed is equivalent to being differentially closed (i.e., a model of DCF$_0$). However, in characteristic $p>0$, the class of separably differentially closed fields strictly contains differentially closed fields. The reason for this is that any differentially closed field $(F,\d)$ of characteristic $p>0$, with field of constants $C_F$, satisfies $C_F=F^p$ \cite{Wood74}, this is referred as $(F,\d)$ being differentially perfect, but there are separably differentially closed fields $(K,\d)$ with $[C_K:K^p]>1$ (as we will see later). In fact, differentially closed fields are characterised as those separably differentially closed fields that are differentially perfect.

\medskip

We now let $(K,\d)$ be a differential field (of arbitrary characteristic). By means of describing separable prime differential ideals of the differential polynomial ring $K\{x\}$ in one variable (see Section \ref{describeideals}), we prove in Section \ref{sepdiff} (see Theorem \ref{foraxioms}) that the class of separably differentially closed fields is elementary with rather elegant first order axioms (in the spirit of Blum's axioms for DCF$_0$ \cite{Blum68}). We also provide several characterisations of being separably differentially closed; in particular, we  give a characterisation in terms of being constrainedly closed in the differential-algebraic sense (see Definition~\ref{defcon} and Theorem~\ref{sevchar}). Furthermore. we also obtain a geometric axiomatisation, in terms of algebraic varieties and their prolongations in the spirit of the geometric axioms for DCF$_0$ \cite{PiPi98} (see Theorem~\ref{geoaxioms}).

The following summarises the several characterisations of being separably differentially closed established in this paper.




\medskip

\noindent {\bf Theorem A.} {\it The following are equivalent}

\begin{itemize}
\item [(i)] $(K,\d)$ is separably differentially closed.

\item [(ii)] For any nonzero $f,g\in K\{x\}$ with $s_f\neq 0$ and $\ord g<\ord f$, there exists $a\in K$ such that $f(a)=0$ and $g(a)\neq 0$.

\item [(iii)] For every differentially algebraic extension $(L,\d)$ of $(K,\d)$, if $L/K$ is separable (as fields) then $(K,\d)$ is existentially closed in $(L,\d)$.

\item [(iv)]$(K,\d)$ is constrainedly closed.

\item [(v)] For each $n$ and every separable prime differential ideal $P$ of $K\{x_1,\dots,x_n\}$, if $g\in K\{x_1,\dots,x_n\}\setminus P$ then there is $\bar a\in K^n$ such that $f(\bar a)=0$ for all $f\in P$ and $g(\bar a)\neq 0$. (Note that this equivalent to: every separable prime differential ideal over $K$ has a zero in $K$.)

\item [(vi)] Let $V$ and $W$ be $K$-irreducible affine varieties over $K$ with $W\subseteq \tau V$, $W$ separable, and $\pi|_{W}:W\to V$ separable. If $O_V$ and $O_W$ are nonempty Zariski-open subsets over $K$ of $V$ and $W$, respectively, then there is a $K$-rational point $a\in O_V$ such that $(a,\d a)\in O_W$.
\end{itemize}

\medskip

Condition (ii) can be written as a scheme of first-order axioms in the language of differential fields, we denote this theory by $\SDCF$. After specifying the characteristic we write $\SDCF_p$. We note that the axioms of DCF$_p$ presented by Wood in \cite{Wood76} are precisely $\SDCF_p$ together with the theory of `differentially perfect fields'.

\medskip

We now note that, for $p>0$, the theory SDCF$_p$ is not complete. We will see in Section \ref{sepdiff} that if $(K,\d)\models \SDCF_p$ then $[K:C_K]=\infty$ and hence $[K:K^p]=\infty$ (recall that $C_K$ contains $K^p$). However, $[C_K:K^p]$ is underdetermined and in fact, we will see that for any $\epsilon\in \NN\cup\{\infty\}$ we can find a model $(K,\d)\models \SDCF_p$ with $[C_K:K^p]=p^\epsilon$ (we call $\epsilon$ the differential degree of imperfection of $(K,\d)$). After specifying $\epsilon$, we denote the theory by $\SDCFe$. Since $\SDCF_{p,0}$ coincides with DCF$_p$, we obtain from \cite{Wood74} that $\SDCF_{p,0}$ is complete. We prove in Section \ref{completions} that this is the case for arbitrary $\epsilon$.

\medskip

\noindent {\bf Theorem B.} {\it For each $p>0$ and $\epsilon\in \NN\cup\{\infty\}$, the theory $\SDCF_{p,\epsilon}$ is complete.}

\medskip

We then go further and in Section \ref{finalprop} introduce the differential analogue of the (algebraic) $\lambda$-functions. After expanding the language and specifying axioms that describe the properties of the differential $\lambda$-functions, we denote the theory by $\SDCFl$ and prove the following.

\medskip

\noindent {\bf Theorem C.} {\it For each $p>0$ and $\epsilon\in \NN\cup\{\infty\}$, the theory $\SDCFl$ admits quantifier elimination, is stable (but not superstable), and prime model extensions exist (and are unique up to isomorphism).}

\medskip

The results in this paper are differential analogues of (some of the) properties of SCF$_p$ and generalise (to arbitrary differential degree of imperfection) the known results for DCF$_p$ in \cite{Wood74,Wood76}. They lay the foundations for the theory $\SDCFl$ and we expect this theory to be further explored in future work.

\

\noindent {\bf Conventions.} For us, tuple means finite tuple. We use $\NN$ to denote the nonnegative integers.

\

Part of the work of this paper appears in the PhD Thesis of the first author \cite{kai23}.

\

We thank an anonymous referee for their detailed comments and useful feedback.

\section{Preliminaries}

In this section we give a brief, and rather direct, presentation of the algebraic, differential-algebraic, and model-theoretic preliminary results that will be needed/useful in later sections.

\subsection{Algebraic preliminaries} Let $K$ be a field of characteristic $p>0$. Recall that $K$ is said to be separably closed if $K$ has no proper separably algebraic extension. Also, recall that a field extension $L/K$ is called separable if $K$ is linearly disjoint from $L^p$ over $K^p$. The following are known to be equivalent:
\begin{itemize}
\item [(i)] $K$ is separably closed.
\item [(ii)] $K$ is existentially closed (as fields) in every separable extension. Namely, if $L/K$ is separable, then every system of polynomial equations (in finitely many variables) over $K$ with a solution in $L$ already has a solution in $K$. (Recall that the actual definition of e.c. involves inequations but for fields it is easy to see that equations suffice.)
\item [(iii)] $K$ is existentially closed in every separably algebraic field extension.
\item [(iv)] For all $f\in K[x]$ with $\frac{df}{dx}\neq 0$, there is $a\in K$ such that $f(a)=0$.
\end{itemize}

\

Let $E$ be an intermediate field of $K/K^p$. The two cases to keep in mind here are when $E=K$ and when $(K,\d)$ is a differential field with $E=C_K$ (the field of constants). Given $A\subseteq E$, the set of $p$-monomials of $A$ is
$$m(A)=\{a_1^{i_1}\cdots a_n^{i_n}: a_1,\dots,a_n\in A \text{ and } 0\leq i_1,\dots,i_n< p\}.$$
We say that $A$ is a $p$-independent subset of $E$ over $K^p$ if the $p$-monomials of $A$ are linearly independent over $K^p$. Similarly, we say that $A$ is a $p$-spanning set of $E$ over $K^p$ if the $p$-monomials of $A$ span $E$ over $K^p$. Finally, $A$ is said to be a $p$-basis for $E$ over $K^p$ if $A$ is $p$-independent and $p$-spanning for $E$ over $K^p$. 

The proofs of the following two lemmas can be found in \cite[IV.7]{Jacobson64} in the case when $E=K$, but the arguments there can easily be adapted to arbitrary $E$.

\begin{lemma}\label{propE1} Let $E$ be an intermediate field of $K/K^p$. If $E/K^p$ is finite, then $[E:K^p]=p^e$ for some $e\in \NN$. Furthermore, this occurs if and only if there is a $p$-basis of $E$ over $K^p$ of (finite) size $e$.
\end{lemma}

\begin{lemma}\label{propE2}
Let $E$ be an intermediate field of $K/K^p$, and $L$ an extension of $K$. Then, the following are equivalent.
\begin{itemize}
\item [(i)] $E$ and $L^p$ are linearly disjoint over $K^p$.
\item [(ii)] Every $p$-independent subset of $E$ over $K^p$ is also a $p$-independent subset of $L$ over $L^p$.
\item [(iii)] There exists a $p$-basis for $E$ over $K^p$ which is a $p$-independent subset of $L$ over $L^p$.
\end{itemize}
\end{lemma}

In case $E=K$, the number $e\in \NN\cup\{\infty\}$ such that $[K:K^p]=p^e$ is called the degree of imperfection of $K$, aka Ershov invariant (note that when $e=\infty$ the expression $[K:K^p]=p^e$ simply means that the degree of $K/K^p$ is infinite). A $p$-basis for $K$ over $K^p$ is simply referred as a $p$-basis for $K$ (similarly for $p$-independent and $p$-spanning). Furthermore, Lemma~\ref{propE2} yields well-known characterisations of when a field extension $L/K$ is separable.

\ 

The following lemma will be needed in the next section.

\begin{lemma}\label{improved}
Let $E$ be an intermediate field of $K/K^p$. Then, every element of $K$ that is separably algebraic over $E$ is in $E$ (in other words, $E$ is separably algebraically closed in $K$).
\end{lemma}
\begin{proof}
Let $a\in K$ be separably algebraic over $E$. Then $a$ is a root of $x^p-a^p\in E[x]$. If $g$ is the minimal polynomial of $a$ over $E$, then $x^p-a^p$ is a multiple of $g$. However, $g$ has no repeated roots, and so $g$ must be of degree one. In other words, $a\in E$.  
\end{proof}

\subsection{Differential algebraic preliminaries}\label{diffpreli} Here $(K,\d)$ is a differential field of arbitrary characteristic. If $(x_{i})_{i\in I}$ is a family of (differential) indeterminates, the differential polynomial ring over $K$ is defined as
$$K\{(x_i)_{i\in I}\}:= K[\d^j x_i: i\in I, j\geq 0]$$
with derivation extending that on $K$ and $\d(\d^j x_i)=\d^{j+1}x_i$. An ideal $I$ of $K\{(x_i)_{i\in I}\}$ is called differential if $\d (I)\subseteq I$. Given a subset $A$ of $K\{(x_i)_{i\in I}\}$ the ideal generated by $A$ (in $K\{(x_i)_{i\in I}\}$) will be denoted by $(A)$, the \emph{differential} ideal generated by $A$ will be denoted by $[A]$, the \emph{radical differential} ideal generated by $A$ will be denoted by $\{A\}$. A prime ideal $P$ of $K\{(x_i)_{i\in I}\}$ is called separable if $Frac(K\{(x_i)_{i\in I}\}/P)$ is a separable field extension of $K$, if in addition $P$ is differential we use the terminology \emph{separable prime differential ideal}.

\medskip

Given a differential field extension $(L,\d)/(K,\d)$ and $A\subseteq L$, we denote by $K\{A\}$ and $K\langle A\rangle$ the differential ring and differential field generated by $A$ over $K$, respectively. An element $a$ of $L$ is called differentially algebraic over $K$ if the family $(\d^j a)_{j\geq 0}$ is algebraically dependent over $K$ (in other words, $a$ is a zero of a nontrivial differential polynomial over $K$). The extension $(L,\d)$ is said to be differentially algebraic over $K$ if each element of $L$ is.

We say that the differential field $(K,\d)$ is \emph{non-degenerate} if for each nonzero differential polynomial $f\in K\{x\}$ (in one variable) there is $a\in K$ such that $f(a)\neq 0$. Seidenberg has characterised non-degeneracy in terms of the degree of $K/C_K$.

\begin{lemma}\cite{Seiden52}\label{chardege}
The differential field $(K,\d)$ is nondegenerate if and only if $[K:C_K]=\infty$.
\end{lemma}

For non-degenerate differential fields, we have the following differential version of the primitive element theorem.

\begin{theorem}\label{diffPET}\cite[Differential Primitive Element Theorem]{Seiden52}
Suppose $(K,\d)$ is non-degenerate. If $L=K\langle a_1,\dots,a_n\rangle$ with each $a_i$ differentially algebraic over $K$ and $L/K$ separable, then there is $a\in L$ such that $L=K\langle a\rangle$.
\end{theorem} 

In \cite[\S III.4]{Kolbook}, Kolchin proves several versions of the differential basis theorem, we state the one that suits our purposes.

\begin{theorem}\cite[\S III.4, Differential Basis Theorem]{Kolbook}\label{diffbasis}
Let $\bar x=(x_1,\dots,x_n)$ be a finite tuple of differential variables. If $P$ is a separable prime differential ideal of $K\{\bar x\}$, then $P$ is finitely generated as a radical differential ideal. In other words, there are $f_1,\dots,f_s\in P$ such that $P=\{f_1,\dots,f_s\}$.
\end{theorem}

Let $\bar \alpha=(\alpha_1,\dots,\alpha_n)$ be an $n$-tuple from a differential field extension of $K$. The defining differential ideal of $\bar \alpha$ over $K$ is the prime differential ideal of $K\{\bar x\}$, where $\bar x=(x_1,\dots,x_n)$, defined as
$$I_\d(\bar \alpha/K)=\{f \in K\{\bar x\}:  f(\bar\alpha)=0\}.$$
A differential specialisation of $\bar \alpha$ over $K$ is an $n$-tuple $\bar \beta$ from some differential field extension of $K$ with $I_\d(\bar \alpha/K)\subseteq I_\d(\bar\beta/K)$. If we have equality of the defining differential ideals we say that $\bar \beta$ is a generic differential specialisation of $\bar \alpha$ over $K$.

\begin{definition}\cite[\S III.10]{Kolbook}
Let $\bar \alpha$ be an $n$-tuple from a differential field extension of $K$. We say that $\bar \alpha$ is constrained over $K$ if $K\langle \bar \alpha\rangle/K$ is separable and there exists $g\in K\{\bar x\}$ with $g(\bar \alpha)\neq 0$ such that for every differential specialisation $\bar \beta$ of $\bar \alpha$ over $K$, with $K\langle \bar\beta \rangle/K$ separable, if $g(\bar \beta)\neq 0$ then $\bar\beta$ is a generic differential specialisation. We call $g$ a constraint of $\bar \alpha$ over $K$.
 \end{definition}
 
 We summarise some basic properties of constrained tuples. 
 
 \begin{lemma}\cite[\S III.10]{Kolbook}\label{specialise}
 Let $\bar a$ be a (finite) tuple from a differential field extension of $K$ with $K\langle \bar a\rangle/K$ separable. Suppose $g\in K\{\bar x\}$ is such that $g(\bar a)\neq 0$. Then, there is a differential specialisation $\bar \alpha$ of $\bar a$ over $K$ such that $\bar \alpha$ is constrained over $K$ with constraint $g$.
 \end{lemma}
 
 \begin{lemma}\cite[\S III.10]{Kolbook}\label{propconstrained} Let $\bar \alpha$ be a constrained tuple over $K$.
\begin{enumerate}
\item If $\bar \beta$ is such that $K\langle \bar \alpha\rangle=K\langle \bar \beta\rangle$, then $\bar \beta$ is constrained over $K$.
\item If $\bar \beta$ is a tuple from $K\langle \bar \alpha\rangle$ and the extension $K\langle \bar \alpha\rangle/K\langle\bar \beta\rangle$ separable, then $\bar \beta$ is constrained over $K$.
\item If $char(K)=p$ then $C_{K\langle \bar \alpha\rangle}$ is separably algebraic over $K\langle \bar \alpha\rangle^p\cdot C_K$.
\end{enumerate} 
 \end{lemma}
 
 \begin{remark}\label{better}
 We note that part (3) of the previous lemma can be improved to the following. If $\bar \alpha$ is constrained over $K$, then $C_{K\langle \bar \alpha\rangle}=K\langle \bar \alpha\rangle^p\cdot C_K$. Indeed, this follows from the Lemma~\ref{improved} (by taking $E=K\langle \bar \alpha\rangle^p\cdot C_K$).
 \end{remark}
 
 \begin{lemma}\label{diffalgentries}
 If $\bar \alpha=(\alpha_1,\dots,\alpha_n)$ is constrained over $K$, then each $\alpha_i$ is differentially algebraic over $K$.
 \end{lemma}
 \begin{proof}
 Since $K\langle\bar \alpha\rangle/K$ is separable, by the Separating Differential Transcendence Basis Theorem \cite[\S II.10]{Kolbook} there exists $\bar \beta=(\beta_1,\dots,\beta_m)$ a differential transcendence basis of $K\langle\bar \alpha\rangle$ over $K$ such that the extension $K\langle\bar \alpha\rangle/K\langle\bar \beta\rangle$ is separable. It suffices to show that $\bar\beta$ is empty. Towards a contradiction, assume it is not empty. By part~(2) of Lemma~\ref{propconstrained}, $\bar\beta$ is also constrained over $K$. Furthermore, since the extension $K\langle\bar \beta\rangle/K\langle\beta_1\rangle$ is separable (as $\bar \beta$ is differentially algebraically independent), part (2) of Lemma~\ref{propconstrained} yields that $\beta_1$ is constrained over $K$. 
 
 Let $g\in K\{x\}$ be a constraint for $\beta_1$. Let $k$ be larger than any $j$ such that $\d^jx$ appears in $g$, and let $P=[\d^k x]$ be the differential ideal of $K\{x\}$ generated by $\d^k x$. Clearly, $P$ is a separable prime differential ideal with $g\notin P$. Thus, setting $a=x+P$ in the extension $Frac(K\{x\}/P)$ we get that $K\langle a\rangle/K$ is separable and $f(a)=0$ and $g(a)\neq 0$.  Since $\beta_1$ is differentially transcendental, its defining differential ideal is zero, and so $I_\d(\beta_1/K)\subseteq I_\d(a/K)$. It follows that $a$ is a generic differential specialisation of $\beta_1$ over $K$; however, $\d^kx \in I_\d(a/K)$, which contradicts the fact that $I_\d(\beta_1/K)=(0)$.
 \end{proof}
 
 \medskip
 
In Section \ref{describeideals}, we will be dealing with the differential polynomial ring in one variable $K\{x\}$, and in this case we write $x_j$ instead of $\d^jx$ and hence identify $K\{x\}=K[x,\d x,\dots]$ with $K[x_0,x_1,\dots]$.

Let $f\in K\{x\}\setminus K$. The order of $f$, denoted $\ord(f)$, is the highest $n$ such that $x_n$ appears in $f$. In this case, we call $x_n$ the leader of $f$. The degree of $f$, denoted $\deg(f)$, is the degree of $x_n$ in $f$ where $x_n$ is the leader of $f$. The rank of $f$ is $\rk(f)=(\ord f, \deg f)$ and we rank differential polynomials lexicographically on $\rk f$.

The separant of $f$, denoted $s_f$, is defined as the partial derivative of $f$ with respect to its leader. One can (uniquely) write $f$ in the form
$$f=\sum_{i=0}^d g_i(x) x_n^i$$
where $x_n$ is the leader of $f$, $d=\deg f$, and $\ord g_i<\ord f$. The differential polynomial $g_d$ is called the initial of $f$, denoted $i_f$. One readily checks, $\rk s_f$ and $\rk i_f$ are strictly smaller than $\rk f$.

\subsection{Model-theoretic preliminaries}\label{modelpreli} We recall the model-theoretic set-up and (some) results of the theories SCF$_{p,e}$ and DCF$_p$. The intention is to give an idea of the tools and notions that we will need to extend to the context of the theory $\SDCFe$ (separably differentially closed fields) in Sections~\ref{completions} and \ref{finalprop}. 

Let $p>0$ and $e\in\NN\cup\{\infty\}$. The theory of separably closed fields of characteristic $p$ and degree of imperfection $e$, SCF$_{p,e}$, is known to be complete \cite{Ershov67} and stable \cite{Wood79}. Furthermore, there is a natural language (expansion of the field language) in which it admits quantifier elimination. For finite $e$, we add symbols $\bar a=(a_1,\dots,a_e)$ for a $p$-basis and, after enumerating the $p$-monomials 
$$m(\bar a)=\{m_0(\bar a),\dots,m_s(\bar a)\} \quad \text{ where } s=p^e-1,$$
we add symbols for the $\lambda$-functions $\lambda_0,\dots,\lambda_s$. Recall that these are definable functions determined by
$$b=\lambda_0(b)^pm_0(\bar a)+\cdots+\lambda_s(b)^pm_s(\bar a)$$
for $b\in K$, where $K$ is any field with $p$-basis $\bar a$. The theory of separably closed fields with $p$-basis $\bar a$ and $\lambda$-functions $\lambda_0,\dots,\lambda_s$ is denoted SCF$_{p,e}^\lambda$.

For infinite $e$, we add countably many function symbols for the following $\lambda$-functions. For each $n\in\NN$, let $s_n=p^n-1$ and fix an enumeration of the $p$-monomials (as functions) 
$$\{m_0,\dots,m_{s_n}\}.$$ 
For each field $K$ with $e(K)=\infty$, the functions $\lambda_{n,i}:K^{n+1}\to K$, for $i=0,\dots,s_n$, are defined by: let $(a_1,\dots,a_n;b)\in K^{n+1}$, if $(a_1,\dots,a_n)$ are $p$-dependent or $(a_1,\dots,a_n;b)$  are $p$-independent, then $\lambda_{n,i}(a_1,\dots,a_n;b)=0$ for $i=0,\dots,s_n$; otherwise, they are (uniquely) determined by
$$b=\lambda_{n,0}(a_1,\dots,a_n;b)^pm_0(\bar a)+\cdots+\lambda_{n,s_n}(a_1,\dots,a_n;b)^pm_{s_n}(\bar a).$$
The theory of separably closed fields of infinite degree of imperfection equipped with these $\lambda$-functions $\lambda_{n,i}$, $n\in \NN, i=0,\dots,s_n$, is denoted SCF$_{p,\infty}^\lambda$.

In either case, $e$ finite or infinite, SCF$_{p,e}^\lambda$ admits quantifier elimination~\cite{Delon88}. 

\

We now move on to the differential context and work in the language of differential fields. A differential field $(K,\d)$ of characteristic $p>0$ is called \emph{differentially perfect} if $C_K=K^p$, where $C_K$ denotes the field of $\d$-constants of $K$. In \cite[\S II.3]{Kolbook} it shown that $(K,\d)$ is differentially perfect if and only if every differential field extension is a separable extension.

The theory of differentially closed fields in characteristic $p>0$, DCF$_p$, is the model-companion of the theory of differential fields of characteristic $p$ and also the model-completion of the theory of differentially perfect fields of characteristic $p$. Furthermore, DCF$_p$ is complete and stable \cite{Wood74,Wood76}.

There is a natural language in which DCF$_p$ admits quantifier elimination. One simply needs to add the $p$-th root function on constants. More precisely, for a differentially perfect $(K,\d)$, we let $r:K\to K$ be defined as
$$
\left\{
\begin{matrix}
\; r(b)=0 & \text{ if }b\notin C_K \\
\quad b= r(b)^p & \text{if } b\in C_K
\end{matrix}
\right.
$$
for any $b\in K$. After adding this function, we denote the theory by DCF$_{p}^r$. In \cite{Wood76}, it is shown that DCF$_p^r$ has quantifier elimination. We note that the function~$r$ can be thought of as a ``differential'' $\lambda$-function (in the case of the differential degree of imperfection zero). This is explained in Section~\ref{qefinite}; in particular, see Remark~\ref{qedcfp}.

\section{On separable prime differential ideals of $K\{x\}$}\label{describeideals}

In this section, we provide a description of those prime differential ideals of $K\{x\}$ that are separable in terms of irreducible elements of $K\{x\}$ with nonzero separant. A similar description but restricted to characteristic zero appears in \cite[II.1]{MMP}. In characteristic zero, the description can also be deduced from \cite[IV.9]{Kolbook}. However, for positive characteristic we are not aware of a detailed reference from which our results could be deduced. 

Throughout this section $(K,\d)$ denotes a differential field (of arbitrary characteristic). Recall that for an ideal $I$ of $K\{x\}$ and $s\in K\{x\}$, the saturated ideal of $I$ over $s$ is defined as
$$I:s^\infty\; =\; \{h\in K\{x\}: s^m h\in I \text{ for some }m\geq 0\}.$$
If $I$ is a differential ideal, one readily checks that $I:s^\infty$ is also a differential ideal. The goal of this section is to prove the following:

\begin{theorem}\label{sepideal} \
\begin{enumerate}
\item If $P$ is a nonzero separable prime differential ideal of $K\{x\}$, then $P=[f]:s_f^\infty$ for any $f$ of minimal rank in $P$ (any such $f$ must be irreducible).
\item If $f\in K\{x\}$ is irreducible with $s_f\neq 0$, then $[f]:s_f^\infty$ is a separable prime differential ideal and $f$ is of minimal rank in it.
\end{enumerate}
\end{theorem}

We prove this after a series of lemmas. We will make use of the following well-known (differential) division algorithm lemma.

\begin{lemma}\cite[\S I.9]{Kolbook}\label{divlemma}
Let $f\in K\{x\}$ be nonzero.
\begin{enumerate}
\item For every $g\in K\{x\}$, there exists $n\geq 0$ and $g_0\in K\{x\}$ with $\ord g_0\leq \ord f$ such that
$$s_f^n\cdot g\equiv g_0\quad \mod [f].$$
\item For every $h\in K\{x\}$ with $\ord h\leq\ord f$, there exists $m\geq 0$ and $h_0$ with $\rk h_0<\rk f$ such that
$$i_f^m\cdot h \equiv h_0 \quad \mod (f).$$
\end{enumerate}
\end{lemma}

We will also need the following technical lemma. Recall that we identify the differential polynomial ring $K\{x\}$ with $K[x_0,x_1,\dots]$ where $x_i$ stands for $\d^i x$.

\begin{lemma}\label{technical}
Let $f\in K\{x\}$ with $\ord f=n$. For each $i=0,1,\dots, n$, we have
$$\frac{\partial (\d f)}{\partial x_i}=\d\left(\frac{\partial f}{\partial x_i}\right)+\frac{\partial f}{\partial x_{i-1}}$$
with the convention that $\frac{\partial f}{\partial x_{-1}}=0$
\end{lemma}
\begin{proof}
We have
$$\d f=f^\d +\sum_{j=0}^n \frac{\partial f}{\partial x_j}\; x_{j+1}$$
where $f^\d$ stands for the differential polynomial obtained by applying $\d$ to the coefficients of $f$. Differentiating both sides with respect to $x_i$ we get
\begin{align*}
\frac{\partial (\d f)}{\partial x_i} & =\frac{\partial f^\d}{\partial x_i}+\sum_{j=0}^n\left(\frac{\partial^2 f}{\partial x_i\partial x_j}\; x_{j+1}+\frac{\partial f}{\partial x_j}\; \frac{\partial x_{j+1}}{\partial x_i}\right) \\
& =\left(\frac{\partial f}{\partial x_i}\right)^\d+\sum_{j=0}^n \frac{\partial^2 f}{\partial x_j\partial x_i}\; x_{j+1}+ \frac{\partial f}{\partial x_{i-1}} \\
&=\d\left(\frac{\partial f}{\partial x_i}\right)+\frac{\partial f}{\partial x_{i-1}}.
\end{align*}
\end{proof}

\begin{lemma}\label{part1}
Let $P$ be a nonzero prime differential ideal of $K\{x\}$ and let $f$ be of minimal rank in $P$.
\begin{enumerate}
\item [(i)] If $s_f=0$, then $\frac{\partial f}{\partial x_i}\in P$ for all $i=0,\dots,\ord f$.
\item[(ii)] If $\frac{\partial f}{\partial x_i}\in P$ for all $i=0,\dots,\ord f$, then $P$ is not separable.
\end{enumerate}
\end{lemma}

\begin{proof}
Let $n=\ord f$.

\smallskip
(i) We proceed by (backwards) induction on $i=n,n-1,\dots,1,0$. For $i=n$, the assertion is true since
$$\frac{\partial f}{\partial x_n}=s_f=0\in P.$$
Now let $n>i>0$ and assume that $\frac{\partial f}{\partial x_i}\in P$. 
From the formula
$$\d f=f^\d +\sum_{j=0}^{n-1} \frac{\partial f}{\partial x_j}\; x_{j+1} + s_f \, x_{n+1},$$
we see that $\ord \d f\leq \ord f$ (since $s_f=0$). By Lemma~\ref{divlemma}(2), there is $m\geq 0$ and $r\in K\{x\}$ with $\rk r<\rk f$ such that
$$i_f^m \d f \equiv  r \quad \mod (f).$$
Since $f$ and $\d f$ are in $P$, we get $r\in P$ and so $r=0$ (since $f$ of minimal rank in $P$). Thus, $i_f^m \d f\in (f)$, but as $f$ is irreducible and $\ord i_f<\ord f$ we actually have $\d f\in (f)$. Hence there is $g\in K\{x\}$ such that $\d f=gf$. Differentiating both sides of the latter equality with respect to $x_i$ and using the formula given in Lemma~\ref{technical}, we have
$$\d\left(\frac{\partial f}{\partial x_i}\right)+\frac{\partial f}{\partial x_{i-1}}=\frac{\partial g}{\partial x_i} \; f+g\; \frac{\partial f}{\partial x_i}.$$
Since $f$, $\frac{\partial f}{\partial x_i}$, and $\d\left(\frac{\partial f}{\partial x_i}\right)$ are all in $P$ (the latter because $P$ is a differential ideal), we get that $\frac{\partial f}{\partial x_{i-1}}\in P$, as desired.

\medskip

(ii) Consider the differential field $L=Frac(K\{x\}/P)$ and let $a=x+P$. We prove that the tuple $(a,\d a,\dots,\d^n a)$ witnesses that $L/K$ is not a separable extension, which means that $P$ is not separable. Let $h\in K\{x\}$ with $\ord h\leq n$ such that $h(a)=0$. It suffices to show that $\frac{\partial h}{\partial x_i}(a)=0$ for all $i=0,\dots,n$. Since $\ord h\leq \ord f$ and $h\in P$, by a similar argument to part (i), we get that there is $g\in K\{x\}$ such that $h=gf$.  Differentiating both sides with respect to $x_i$ we get
$$\frac{\partial h}{\partial x_i}=\frac{\partial g}{\partial x_i}\; f + g\; \frac{\partial f}{\partial x_i}.$$
Evaluating at $a$, the right-hand-side vanishes (as $\frac{\partial f}{\partial x_i}\in P$). Thus, $\frac{\partial h}{\partial x_i}(a)=0$.
\end{proof}

\begin{lemma}\label{equalP}
Let $P$ be a prime differential ideal of $K\{x\}$. If there is $f\in P$ of minimal rank with the property that $s_f\neq 0$, then $P=[f]:s_f^\infty$.
\end{lemma}
\begin{proof}
Since $\rk s_f<\rk f$ and $f$ is of minimal rank, $s_f\notin P$. Then, since $P$ is prime, we have $[f]:s_f^\infty\subseteq P$. To prove the other containment, suppose $g\in P$. By Lemma~\ref{divlemma}(1), there are $n\geq 0$ and $g_0\in K\{x\}$ with $\ord g_0<\ord f$ such that
\begin{equation}\label{sepeq}
s_f^n g\equiv g_0 \quad \mod [f].
\end{equation}
Since $g\in P$, we get $g_0\in P$. In case $\ord g_0<\ord f$, then $g_0=0$ (by minimal rank of $f$), and so $g\in [f]:s_f^\infty$. Otherwise, $\ord g_0=\ord f$ and then by Lemma~\ref{divlemma}(2) there are $m\geq 0$ and $g_1$ with $\rk g_1<\rk f$ such that
$$i_f^m g_0\equiv g_1 \quad (f).$$
Since $g_0\in P$, we get $g_1\in P$. By minimality of $f$, we get $g_1=0$. Thus, $i_f g_0\in (f)$. As $\ord i_f<\ord f$, we get $i_f\notin P$, and so, as $(f)$ is prime, we obtain $g_0\in (f)\subseteq [f]$. It follows from congruence \eqref{sepeq} above that $g\in [f]:s_f^\infty$.  
\end{proof}

We can now easily prove the part (1) of Theorem~\ref{sepideal}.

\begin{proof}[Proof of Theorem~\ref{sepideal}(1)]
Let $f$ be of minimal rank in $P$. Since $P$ is separable, Lemma~\ref{part1} tells us that $s_f\neq 0$. The assertion now follows from Lemma~\ref{equalP}.  
\end{proof}

We now move on to prove part (2) of the theorem. We first need a lemma. 

\begin{lemma}\label{otherway}
Let $f\in K\{x\}$ be irreducible with $s_f\neq 0$.
\begin{enumerate}
\item If $g\in [f]$ with $\ord g\leq \ord f$, then $g\in (f)$.
\item $[f]:s_f^\infty$ is a prime differential ideal.
\end{enumerate}
\end{lemma}
\begin{proof}
(1) Let $g\in [f]$ be of order at most $\ord f$. Then there is $m$ and $h_0,\dots,h_m\in K\{x\}$ such that
\begin{equation}\label{reduce}
g=h_0 f +h_1\d f+\cdots+ h_m\d^m f.
\end{equation}
We prove the result by induction on $m$. If $m=0$, then $g=h_0f$ and so $g\in (f)$. Now assume it holds for $m-1\geq 0$. Let $n=\ord f$. The Leibniz rule for derivations implies that there is $f_m$ with $\ord f_m< n+m$ such that
$$\d^m f=s_f \; x_{n+m}+f_{m}.$$
Since $s_f\neq 0$, we can work in the localisation $K\{x\}_{s_f}$ and we can specialise $x_{n+m}$ to $-f_m\cdot s_f^{-1}$ in equality \eqref{reduce}. This yields 
$$g=h'_0f +h'_1 \d f+\cdot +h'_{m-1}\d^{m-1}f.$$
for some $h'_0,\dots,h'_{m-1}\in K\{x\}_{s_f}$. In fact, there is $m\geq 0$ such that $s_f^m h'_i\in K\{x\}$ for all $i=0,\dots,m-1$. By induction, it follows that $s_f^m g\in (f)$. Since $\rk s_f < \rk f$ and $(f)$ is prime, we obtain $g\in (f)$.

\smallskip
(2) We prove $[f]:s_f^\infty$ is prime. Let $g,h\in K\{x\}$ such that $g\cdot h\in [f]:s_f^\infty$. By Lemma~\ref{divlemma}, there are $n,m$ and $g_0,h_0\in K\{x\}$ with $\ord(g_0)$ and $\ord h_0$ at most $\ord f$ such that
\begin{equation}\label{forprime}
s_f^n g\equiv g_0 \; \mod [f] \quad \text{ and }\quad s_f^m h\equiv h_0 \; \mod [f].
\end{equation}
Then
$$s_f^{n+m}gh \equiv g_0 h_0 \; \mod [f].$$
Since $gh\in [f]:s_f^\infty$, there $r$ such that $s_f^{r}gh\in [f]$. Then, by the above display, we get $s_f^rg_0h_0\in [f]$. Since $\ord (s_f^rg_0h_0)\leq \ord f$, part (1) gives us that $s_f^rg_0h_0\in (f)$. Since $\rk s_f<\rk f$ and $(f)$ is prime, we get that either $g_0$ or $h_0$ is in $(f)$. By the congruences in \eqref{forprime}, we get that either $s_f^n g\in [f]$ or $s_f^m h\in [f]$, as desired.
\end{proof}

We can now prove part (2) of Theorem~\ref{sepideal}.

\begin{proof}[Proof of Theorem~\ref{sepideal}(2)]
Let $P=[f]:s_f^\infty$. By Lemma~\ref{otherway}(2), $P$ is a prime differential ideal. We now prove $f$ is of minimal rank in $P$. Let $g\in P$ with $\ord(g)\leq \ord f$. Then for some $n$ we have $s_f^n g\in [f]$. Since $\ord(s_f^n g)\leq \ord f$, Lemma~\ref{otherway}(1) gives $s_f^n g\in (f)$. Since $\rk s_f<\rk f$ and $(f)$ is prime, we get $g\in (f)$. So, $\deg f\leq \deg g$, which implies $\rk f\leq \rk g$.

All that remains to show is that $P$ is separable. Equivalently, we show that the differential field $L=Frac(K\{x\}/P)$ is separable over $K$. Let $a=x+P$. Let $n=\ord f$. Since $s_f(a)\neq 0$, one readily checks that
$$K\langle a\rangle =K(a,\d a,\dots,\d^n a).$$
We claim that $a,\d a,\dots,\d^{n-1}a$ forms a separating transcendence basis for $K\langle a\rangle$ (and this is enough as then $L/K$ will be separable). Indeed, if $h\in K[x_0,\dots,x_{n-1}]$ with $h(a,\dots,\d^{n-1}a)=0$ then $h\in P$, as $f$ is minimal in $P$ we get $h$ must be zero. Thus, $a,\d a,\dots,\d^{n-1}a$ are algebraically independent over $K$. To finish, we prove that $K\langle a\rangle$ is separably algebraic over $K(a,\dots,\d^{n-1}a)$. By the above displayed equality, it suffices to show that $\d^n a$ is separable over $K(a,\cdots,\d^{n-1}a)$. Consider the polynomial $p(x_n)=f(a,\cdots,\d^{n-1}a, x_n)\in K(a,\dots,\d^{n-1}a)[x_n]$. Then $p(a)=0$ and 
$$\frac{d p}{d x_n}(\d^{n}a)=s_f(a)\neq 0.$$
Hence, $\d^{n}a$ is indeed separable over $K(a,\cdots,\d^{n-1}a)$.
\end{proof}

We point out the following two useful consequences.

\begin{corollary}
Let $P$ be a nonzero prime differential ideal of $K\{x\}$. Then, $P$ is separable if and only if among the elements of $P$ of minimal rank there is one with nonzero separant.
\end{corollary}
\begin{proof}
$(\Rightarrow)$ Suppose $P$ is separable. Let $f$ be irreducible of minimal rank in $P$. By Lemma~\ref{part1}, $s_f$ is nonzero. 
\smallskip

$(\Leftarrow)$ Let $f$ be an element of $P$ of minimal rank and suppose that $s_f\neq 0$. By Lemma~\ref{equalP}, $P$ equals $[f]:s_f^\infty$, and, by Theorem~\ref{sepideal}(2), the latter ideal is separable.
\end{proof}

\begin{corollary}\label{isotype} \
\begin{enumerate}
\item If $f\in K\{x\}$ is irreducible with nonzero $s_f$, then $[f]:s_f^\infty$ is the unique prime differential ideal of $K\{x\}$ containing $f$ but not containing any $g\in K\{x\}$ with $\ord g<\ord f$.
\item If $a$ is constrained over $K$, then there are $f,g\in K\{x\}$ with $\ord g<\ord f$ such that the differential isomorphism type of $a$ over $K$ is determined by $f(x)=0\land g(x)\neq 0$.
\end{enumerate}
\end{corollary}
\begin{proof}
(1) Suppose $Q$ is a prime differential ideal of $K\{x\}$ containing $f$ but not any $g$ with $\ord g<\ord f$. By Lemma~\ref{equalP}, it suffices to prove that $f$ is of minimal rank in $Q$. Let $h$ be of minimal rank in $Q$ (note that $\ord h=\ord f$). Towards a contradiction, suppose $\deg h<\deg f$. By the division algorithm, there is $n$ such that $i_h^n \cdot f\in (h)$. Since $\ord i_h< \ord h$, the latter implies that $h$ is a factor of $f$, which contradicts the irreducibility of $f$. 

\medskip
(2) The differential defining ideal of $a$ over $K$, denoted $I_\d(a/K)$, is a separable prime differential ideal of $K\{x\}$, and thus, by Theorem~\ref{sepideal}, $I_\d(a/K)=[f]:s_f^{\infty}$ for some irreducible $f$ and $s_f\neq 0$. Let $q\in K\{x\}$ be a constraint of $a$ over $K$. Towards a contradiction, suppose there is no $g\in K\{x\}$ with $\ord g<\ord f$ such that $g$ is a constraint for $a$ over $K$. Consider
$$\Phi(x)=\{f(x)=0\land q(x)=0\}\cup\{g(x)\neq 0: g\in K\{x\} \text{ and }\ord g<\ord f\}.$$
A (first-order) compactness argument yields that $\Phi(x)$ is satisfiable in some differential field extension of $K$, say by an element $b$. Then, by part (1), $I_\d(b/K)=I_\d(a/K)$. But then, as $q(b)=0$, we get $q(a)=0$ and this contradicts the fact that $q$ is a constraint for $a$ over $K$.
\end{proof}

\section{Separably differentially closed fields}\label{sepdiff}

In this section, we introduce the class of separably differentially closed fields. We derive some basic properties and prove that this class is elementary (in the first-order sense in the language of differential fields). Furthermore, we provide several characterisations; in particular, we prove that being separably differentially closed is equivalent to being constrainedly closed in the sense of Kolchin (with the right adaptation in positive characteristic). We also give an algebro-geometric characterisation in the spirit of the geometric axioms for DCF$_0$. Unless otherwise stated, fields in this section are of arbitrary characteristic.

\medskip

 Let $\Ld$ be the language of differential fields; namely, $\Ld=\{0,1,+,-,\cdot,^{-1},\d\}$. Recall that, given a differential field extension $(L,\d)/(K,\d)$, we say that $(K,\d)$ is existentially closed in $(L,\d)$ if for every quantifier-free $\Ld(K)$-formula $\phi(\bar v)$ we have
 $$L\models \exists \bar v \; \phi(\bar v) \quad \implies \quad K\models \exists \bar v\; \phi(\bar v).$$
 
 \medskip

\begin{definition}
A differential field $(K,\d)$ is said to be separably differentially closed if for every differential field extension $(L,\d)/(K,\d)$ the following holds: if the field extension $L/K$ is separable, then $(K,\d)$ is existentially closed in $(L,\d)$.
\end{definition}

\begin{remark} \
\begin{enumerate}
\item [(i) ]The above definition takes its cue from the field-theoretic notion of being separably closed. Indeed, recall that a field $K$ is separably closed if and only if for every separable field extension $L/K$ we have that $K$ is existentially closed in $L$ (in the language of fields).
\item [(ii)] Under the assumption that $(K,\d)$ is differentially perfect (and so every differential field extension is separable), we see from the definition that $(K,\d)$ being separably differentially closed is equivalent to being differentially closed (i.e., existentially closed in every differential extension \cite{Wood73}).
\end{enumerate}
\end{remark}

One of the goals of this section is to prove that the class of separably differentially closed fields is elementary (in the language $\Ld$). The first-order axioms below are in the manner of Blum's axioms for $\DCF_0$ \cite{Blum68} and Wood's axioms for $\DCF_p$ \cite{Wood76}. This is the content of the following theorem.

\begin{theorem}\label{foraxioms}
Let $(K,\d)$ be a differential field. The following are equivalent:
\begin{enumerate}
\item $(K,\d)$ is separably differentially closed
\item For any nonzero $f,g\in K\{x\}$ with $s_f\neq 0$ and $\ord g<\ord f$, there exists $a\in K$ such that $f(a)=0$ and $g(a)\neq 0$.
\end{enumerate}
\end{theorem}

\begin{proof}
$(1)\Rightarrow (2)$. Let $f,g$ nonzero elements of the differential polynomial ring 
$$K\{x\}=K[x_0,x_1,x_2,\dots]$$
with $s_f\neq 0$ and $\ord g<\ord f$. Since $K\{x\}$ is a UFD, we can factorise $f$ into its irreducible factors; namely, $f=f_1\cdots f_s$. Let $n=\ord f$. Since $s_f\neq 0$ and 
$$s_f=\frac{\partial f}{\partial x_n}=\frac{\partial f_1}{\partial x_n}\; f_2\cdots f_s+\;\cdots\;+f_1\cdots f_{s-1}\;\frac{\partial f_s}{\partial x_n}$$ 
at least one $\frac{\partial f_i}{\partial x_n}$ is nonzero (and so $s_{f_i}\neq 0$ and $\ord g<\ord f_i$). Thus, we may assume (without loss of generality) that $f$ itself is irreducible, By Theorem~\ref{sepideal}, if we set $P=[f]:s_f^\infty$, then $P$ is a separable prime differential ideal of $K\{x\}$. Let $L=$Frac$(K\{x\}/P)$ and $b=x+P$. Then $L/K$ is separable and $f(b)=0$. Since $\ord g<\ord f$, by Theorem~\ref{sepideal}  $g\notin P$, and so $g(b)\neq 0$. In other words,
$$L\models \exists v \left( f(v)=0 \land g(v)\neq 0\right).$$
Since $(K,\d)$ is assumed to be separably differentially closed and $L/K$ is a separable field extension, there is $a\in K$ such that $f(a)=0$ and $g(a)\neq 0$. 

\medskip

$(2)\Rightarrow (1)$ We first note that the current assumption yields that $(K,\d)$ is non-degenerate. Indeed, for any nonzero $g\in K\{x\}$, let $m=\ord g$ and $f=\d^{m+1}x$. Then, $s_f\neq 0$ and $\ord g<\ord f$, and so by the assumption there is $a\in K$ such that $f(a)=0$ and $g(a)\neq 0$. The latter shows non-degeneracy.

Now, to prove that $(K,\d)$ is separably differentially closed, assume that $(L,\d)$ is a differential extension with $L/K$ separable and $\phi(\bar x)$ is a quantifier-free $\Ld(K)$-formula with $\bar x=(x_1,\dots,x_s)$ such that
$$L\models \exists \bar x \phi(\bar x).$$
That is, there is $\bar b$ from $L$ such that $L\models \phi(\bar b)$. Since $\phi$ is quantifier-free, we may assume it is of the form
$$f_1(\bar x)=0\land \cdots\land f_r(\bar x)=0\land h(\bar x)\neq 0.$$
where $f_1,\dots,f_r,h\in K\{\bar x\}$. Since $K\langle\bar b \rangle/K$ is separable, by Lemma~\ref{specialise} there exists a $K$-differential specialisation $\bar \alpha$ of $\bar b$ such that $\bar \alpha$ is constrained over $K$ with constraint $h$. It then follows that
$$K\langle \bar \alpha\rangle \models \phi(\bar\alpha).$$
To finish the proof it suffices to show that $\bar \alpha$ is a tuple from $K$. By Lemma~\ref{diffalgentries}, the extension $K\langle \bar \alpha\rangle/K$ is differentially algebraic (and we know it is also separable, by definition of constrained tuple), and so, since $(K,\d)$ is non-degenerate, Theorem~\ref{diffPET} yields a single element $a\in K\langle \bar \alpha\rangle$ such that 
$$K\langle \bar \alpha\rangle =K\langle a\rangle.$$
It is thus enough to show that $a\in K$. Note that, by Lemma~\ref{propconstrained}(1), $a$ is constrained over $K$. By Corollary~\ref{isotype}(2), there are $f,g\in K\{x\}$ with $\ord g<\ord f$ such that the differential $K$-isomorphism type of $a$ is determined by $f(x)=0 \land g(x)\neq 0$. By the assumption, there is a solution of this differential system in $K$. It follows that $a\in K$, as desired. 
\end{proof}

Condition (2) in Theorem \ref{foraxioms} can be written as a scheme of first-order sentences in the language $\Ld$. We denote this theory $\SDCF$. Thus, for a differential field $(K,\d)$ we have that being separably differentially closed is equivalent to $K\models \SDCF$. Once we specify the characteristic $p$ (zero or prime), we write this theory as $\SDCF_p$.

\begin{remark}\
\begin{enumerate}
\item [(i)] In characteristic zero, the condition $s_f\neq 0$ in Theorem~\ref{foraxioms}(2) is always satisfied. Thus, we recover Blum's axioms and indeed the models of $\DCF_0$ and $\SDCF_{0}$ coincide.
\item [(ii)] In characteristic $p>0$, if we add the \emph{differential perfectness axiom}
$$\forall x \exists y (\d x=0\rightarrow x=y^p)$$
to $\SDCF_p$ we recover Wood's axioms for the theory $\DCF_p$.
\end{enumerate}
\end{remark}

We now prove some basic properties of models of $\SDCF$. In particular, in positive characteristic, they are all separably closed of infinite degree of imperfection. 

\begin{lemma}\label{basicprop} Let $K\models \SDCF_p$ with $p>0$.
\begin{enumerate}
\item [(i)] The extension $K/C_K$ is of infinite degree.
\item [(ii)] $K\models SCF_{p,\infty}$
\item [(iii)] $C_K\models SCF_{p,\infty}$
\end{enumerate}
\end{lemma}
\begin{proof}
(i) At the beginning of the proof of (2) of Theorem~\ref{foraxioms} we argued that $(K,\d)$ is non-degenerate. The result now follows from Lemma~\ref{chardege}. 

(ii) By the axioms, it follows that any separable polynomial over $K$ has a root in $K$, and thus $K$ is separably closed. By part (i), $[K:C_K]$ is infinite. Since $K^p$ is a subfield of $C_K$, it follows that $[K:K^p]$ is also infinite.

(iii) That $C_K$ is separably closed follows from (ii) using the general fact that $C_K$ is separably algebraically closed in $K$. Now, since $[K:C_K]$ is infinite (by non-degeneracy), we have that $[K^p:C_K^p]$ is also infinite (by applying Frobenius morphism) and thus $[C_K:C_K^p]$ is also infinite.
\end{proof}

We now aim to exhibit several characterisations of being separably differentially closed. One of them in terms of being constrainedly closed. 

\begin{definition}\label{defcon}
A differential field $(K,\d)$ is said to be constrainedly closed if for every (finite) tuple $\bar a$, from a differential field extension, the following holds: if $\bar a$ is constrained over $K$ then each entry of $\bar a$ is in $K$.
\end{definition} 

We point out that in \cite[\S 2-3]{Kol74} Kochin considered constrainedly closed differential fields in \emph{characteristic zero}. Our definition here differs from his (in positive characteristic) by a subtle point. A differential field extension of $K$ is said to be a constrained extension if every (finite) tuple from that field is constrained over $K$. Kolchin then defines constrainedly closed as:
$$(\dagger) \quad K \text{ has no proper constrained extension.}$$
It turns out that in characteristic zero our definition coincides with Kochin's ($\dagger$). This is an immediate consequence of the following fact.

\begin{fact}\label{conext}\cite[Proposition 1]{Kol74}
Suppose $(K,\d)$ is of characteristic zero. If $\bar a$ is a tuple constrained over $K$, then any tuple in $K\langle \bar a\rangle$ is also constrained over $K$.
\end{fact}

However, this fact does {\bf not} hold in positive characteristic. Consider the field $\mathbb F_p$ with $p$ elements and the function field $\mathbb F_p(t)$ with derivation $\d(t)=1$. Then, $t$ is constrained over $\mathbb F_p$; however, $t^p$ is not (since $\d(t^p)=0$). 

Now, not only does Fact~\ref{conext} fail in characteristic $p>0$, but, as the following lemma shows, condition ($\dagger$) is too weak of a condition. 

\begin{lemma}
A differential field $(K,\d)$ of characteristic $p>0$ satisfies $(\dagger)$ if and only if $K$ is separably closed.
\end{lemma}
\begin{proof}
($\Rightarrow$) This follows from the fact that every separably algebraic extension of $K$ is a constrained extension.

($\Leftarrow$) Let $(L,\d)$ be a constrained extension and $a\in L$ (in particular, $a$ is constrained over $K$). Towards a contradiction, assume $a\notin K$. Since $K$ is separably closed and $K\langle a\rangle$ is separable over $K$, we get that $a$ is transcendental over $K$. Hence, so is $a^p$. But, since $\d(a^p)=0$, $a^p$ cannot be constrained over $K$. This contradicts the fact that $L$ is a constrained extension.
\end{proof}

In particular, any separably closed field of characteristic $p>0$ equipped with the trivial derivation satisfies ($\dagger$). All this discussion is to argue that condition ($\dagger$) used by Kolchin in characteristic zero is not the right one to define constrainedly closed in positive characteristic. The following theorem justifies that our definition here (namely,  Definition~\ref{defcon}) seems to be the correct one.

\begin{theorem}\label{sevchar}
Let $(K,\d)$ be a differential field (arbitrary characteristic). The following are equivalent:
\begin{enumerate}
\item $(K,\d)$ is separably differentially closed.
\item for every differentially algebraic extension $(L,\d)$ of $(K,\d)$, if $L/K$ is separable (as fields) then $(K,\d)$ is existentially closed in $(L,\d)$.
\item $(K,\d)$ is constrainedly closed (in the sense of Def.\ref{defcon}).
\item for each $n$ and every separable prime differential ideal $P$ of $K\{x_1,\dots,x_n\}$, if $g\in K\{x_1,\dots,x_n\}\setminus P$ then there is $\bar a\in K^n$ such that $f(\bar a)=0$ for all $f\in P$ and $g(\bar a)\neq 0$.
\end{enumerate}
\end{theorem}

\begin{proof}
$(1)\Rightarrow (2)$ Immediate from the definition.

$(2)\Rightarrow (3)$ Let $\bar a$ be a constrained tuple over $K$. Then, $K\langle \bar a\rangle/K$ is separable and, by Lemma~\ref{diffalgentries}, $K\langle \bar a\rangle$ is a differentially algebraic extension of $K$. Let $g$ be a constraint for $\bar  a$. By Theorem~\ref{diffbasis}, the differential defining ideal of $a$ over $K$ is finitely generated as a radical differential ideal; that is, $I_\d(\bar a/K)=\{f_1,\dots,f_s\}$. Consider the system
$$f_1=0\land \cdots \land f_s=0\land g\neq 0.$$
By the assumption, there is $\bar b$ from $K$ satisfying this system, but then $I_\d(\bar a/K)=I_\d(\bar b/K)$. It then follows that $\bar a$ is from $K$, and hence $K$ is constrainedly closed.

$(3)\Rightarrow (4)$ Let $\bar x=(x_1,\dots,x_n)$. Consider the differential field $Frac(K\{\bar x\}/P)$ and let $\bar a=\bar x+P$. Then, $K\langle \bar a\rangle/K$ is separable, $P=I_\d(\bar a/K)$ and $g(\bar a)\neq 0$. By Lemma~\ref{specialise}, there is a $K$-differential specialisation $\bar b$ of $\bar a$ such that $\bar b$ is constrained and $g(\bar b)\neq 0$. In particular, $P\subseteq I_\d(\bar b/K)$. Since $(K,\d)$ is constrainedly closed, we get $\bar b$ is from $K$ and satisfies the desired conditions.

$(4)\Rightarrow (1)$ Let $f,g\in K\{x\}$ be nonzero with $s_f\neq 0$ and $\ord g<\ord f$. By Theorem~\ref{foraxioms}, it suffices to prove that there is $a\in K$ with $f(a)=0\land g(a)\neq 0$. Since $f$ must have an irreducible factor of the same order as $f$ and with nonzero separant, we may assume that $f$ is already irreducible. By Theorem~\ref{sepideal}, $P=[f]:s_f^\infty$ is a separable prime differential ideal of $K\{x\}$ and $g\notin P$. Now the assumptions yields the desired $a\in K$.
\end{proof}

Lastly, we provide a geometric axiomatisation for $\SDCF$ in the spirit of the Pierce-Pillay axioms for DCF$_0$ \cite{PiPi98}. We note that such geometric axioms for DCF$_p$, $p>0$, appear in \cite[\S 3]{Gogo21}. Let $(K,\d)$ a differential field of arbitrary characteristic. To formulate the geometric axioms we recall that given an algebraic variety $V$ over $K$ there exists an algebraic bundle $\pi:\tau V\to V$ over $K$ called that prolongation of $V$ that has the following characteristic property: for any differential field extension $(L,\d)$ of $(K,\d)$ if $a\in V(L)$ then $(a,\d a)\in \tau V(L)$. In the case when $V$ is affine, say $K[V]=K[x_1,\dots,x_n]/I$ with $I$ a radical ideal, the following equations define the prolongation
$$f(\bar x)=0\quad \text{ and }\quad \sum_{i=1}^n\frac{\partial f}{\partial x_i}(\bar x)\cdot y_i+f^\d(\bar x)=0,$$
as $f$ varies in generators of $I$ (the defining ideal of $V$). Also recall that, in case $V$ is $K$-irreducible, we say that $V$ is separable if the function field $K(V)$ is separable over $K$. If $W$ is another $K$-irreducible variety and $\phi:W\to V$ is a morphism (over $K$), we say that $\phi$ is separable if it is dominant and the function field $K(W)$ is separable over the function field $K(V)$ (or rather its isomorphic copy $\phi^*(K(V))$).

\begin{theorem}\label{geoaxioms}
Let $(K,\d)$ be a differential field (of arbitrary characteristic). Then, the following are equivalent.
\begin{enumerate}
\item $K\models \SDCF$
\item Let $V$ and $W$ be $K$-irreducible affine varieties over $K$ with $W\subseteq \tau V$, $W$ separable, and $\pi|_{W}:W\to V$ separable. If $O_V$ and $O_W$ are nonempty Zariski-open subsets over $K$ of $V$ and $W$, respectively, then there is a $K$-rational point $a\in O_V$ such that $(a,\d a)\in O_W$
\end{enumerate}
\end{theorem}
\begin{proof}
(1)$\Rightarrow$(2) Let $(a,b)$ be a $K$-generic point of $W$. By the assumptions, $a$ is a $K$-generic of $V$, and $K(a,b)/K$ and $K(a,b)/K(a)$ are separable extensions. Furthermore, since $W\subseteq \tau V$, there is a derivation $\d:K(a)\to K(a,b)$ extending that in $K$ such that $\d(a)=b$. Since $K(a,b)/K(a)$ is separable, we can further extend the derivation to $\d:K(a,b)\to K(a,b)$ (See for instance, Jacobson \cite[IV.7]{Jacobson64}).

Then, in the differential extension $(K(a,b),\d)$ of $(K,\d)$ we can find a solution to $x\in O_V \land (x,\d x)\in O_W$ (namely, the tuple $a$). Since $K(a,b)/K$ is separable, the fact that $(K,\d)$ is separably differentially closed tells us that we can find a solution in $K$. 

\smallskip

(2)$\Rightarrow$(1) Let $f,g\in K\{x\}$ nonzero with $s_f\neq 0$ and $\ord g<\ord f$. We must find a solution to $f(x)=0\land g(x)\neq 0$ in $K$. We may assume $f$ is irreducible in $K\{x\}$. Then $P=[f]:s_f^\infty$ is a separable prime ideal (by Theorem~\ref{sepideal}). Let $a=x+P$ in the fraction field of $K\{x\}/P$. Letting $n=\ord f$, we see that $(a,\d a,\dots,\d^{n-1}a)$ is algebraically independent over $K$ and $a$ is separably algebraic over $K(a,\dots,\d^{n-1}a)$. In particular, $K(a,\dots,\d^n a)$ is separable over $K$. Now let $c=(a,\dots,\d^{n-1} a)$, $V=loc(c/K)$ and $W=loc(c,\d c/K)$. It follows that $K(W)/K$ and $K(W)/K(V)$ are both separable extensions. Since $\ord g<\ord f$ and $g(a)\neq 0$, we see that the Zariski-open set $O_V:=V\setminus Z(g)$ is nonempty (here $Z(g)$ denotes the vanishing of $g$ as a polynomial in $n$ variables). We can thus apply the geometric assumptions and obtain a tuple $b=(b_0,\dots,b_{n-1})$ from $K$ such that $b\in O_V$ and $(b,\d b)\in W$. The first condition yields $g(b_0)\neq 0$, while the second condition yields $f(b_0)=0$; thus $b_0\in K$ is the desired point.
\end{proof}

\section{Completions of the theory $\SDCF_{p}$}\label{completions}

In this section, we describe the completions of the theory SDCF$_p$. Recall that $\SDCF_p$ denotes the theory of separably differentially closed fields of characteristic $p$. When $p=0$ this theory coincides with DCF$_0$ which we know is complete (among many other things, see \cite[\S 2]{MMP}). However, for $p>0$, SDCF$_p$ is not complete. Recall, from Lemma~\ref{basicprop}, that if $(K,\d)\models \SDCF_p$ then $[K:C_K]=\infty$ and $[K:K^p]=\infty$. However, $[C_K:K^p]$ is underdetermined; and in fact we will see that for any $\epsilon\in \NN\cup\{\infty\}$ we can find a model $(K,\d)\models \SDCF_p$ with $[K:C_K]=p^\epsilon$ (we call $\epsilon$ the differential degree of imperfection of $(K,\d)$). Furthermore, we will prove that the degree $[K:C_K]$ determines the completions; in other words, if $(K,\d)$ and $(L,\d)$ are models of $\SDCF_p$ with $[C_K:K^p]=[C_L:L^p]$ then $K$ and $L$ are elementarily equivalent in the language of differential fields.

\smallskip

Our arguments are modelled after the ones used in SCF$_p$ \cite{Ershov67, Wood79}. That is, we first expand the language to obtain a model-completeness and an amalgamation result, from which completeness will follow. But first, we investigate the differential analogues of a degree of imperfection, $p$-independency, and $p$-basis.

\medskip

For the remainder of this section, we assume $p>0$.

\subsection{Differential degree of imperfection} Let $(K,\d)$ be a differential field of characteristic $p>0$. By Lemma~\ref{propE1}(1) in the preliminaries, there is $\epsilon\in \NN\cup \{\infty\}$ such that
$$[C_K:K^p]=p^\epsilon$$
where $p^{\infty}$ just means that $[C_K:K^p]$ is infinite. We call $\epsilon$ the \emph{differential degree of imperfection of} $(K,\d)$. Sometimes we write this by $\epsilon=\epsilon(K)$. 

\begin{definition} Recall that $\Ld$ denotes the language of differential fields. Let $\epsilon\in \NN\cup \{\infty\}$
\begin{enumerate}
\item By DF$_p$ we mean the theory of differential fields of characteristic $p$.
\item Clearly the condition (on a differential field) of having differential degree of imperfection $\epsilon$ can be expressed as first-order axioms in the language $\Ld$. After adding these axioms to DF$_p$, we denote the new theory by DF$_{p,\epsilon}$.
\item Similarly, the theory of separably differentially closed fields in characteristic $p$ of differential degree of imperfection $\epsilon$ is denoted $\SDCFe$.
\end{enumerate}
\end{definition}

\begin{remark}
Having a differential degree of imperfection $\epsilon(K)=0$ is equivalent to $C_K=K^p$; in other words, $(K,\d)$ is differentially perfect. We thus see that $\SDCF_{p,0}$ coincides with the theory DCF$_p$ of differentially closed in characteristic $p$ \cite{Wood73}.
\end{remark}

Recall that for $A\subseteq K$, the set of $p$-monomials of $A$ is
$$m(A)=\{a_1^{i_1}\cdots a_n^{i_n}:a_1,\dots,a_n\in A \text{ and }0\leq i_1,\dots,i_n<p \}.$$
We say that $A$ is a \emph{differentially $p$-independent} subset of $(K,\d)$ if $A\subseteq C_K$ and the $p$-monomials of $A$ are linearly independent over $K^p$. Similarly, we say that $A$ is a \emph{differentially $p$-spanning} set of $(K,\d)$ if $A\subseteq C_K$ and the $p$-monomials of $A$ span $C_K$ over $K^p$. Finally, $A$ is said to be a \emph{differential $p$-basis} for $(K,\d)$ if $A$ is differentially $p$-independent and a differentially $p$-spanning set for $(K,\d)$. Note that when $C_K=K^p$ (i.e., $(K,\d)$ is differentially perfect), the empty set is a differential $p$-basis for $(K,\d)$.

\begin{remark} \
\begin{enumerate}
\item For $\epsilon\in \NN\cup \{\infty\}$, the differential field $(K,\d)$ has differential degree of imperfection $\epsilon$ if and only if it has a differential $p$-basis of size $\epsilon$. This is a consequence of Lemma~\ref{propE1}(2).
\item If the derivation on $K$ is trivial (namely, $\d\equiv 0$) or equivalently $C_K=K$, a differential $p$-basis for $(K,\d)$ coincides with an (algebraic) $p$-basis for $K$. In particular, the differential degree of imperfection of $(K,\d)$ would coincide with the field-theoretic degree of imperfection of $K$.
\end{enumerate}
\end{remark}

Part (2) of the above remark gives the existence of differential fields with arbitrary differential degree of imperfection. Indeed, simply take $\mathbb F_p(t_1,\dots,t_\epsilon)$ with the trivial derivation. For what remains of this section we prove several basic results on differential $p$-basis and separability that will be useful later. First,  we recall the following well-known fact on extending derivations in fields of positive characteristic.

\begin{fact}\cite[Lemma 1(i)]{Wood73}\label{extendder}
Let $b\in C_K$ and $c$ a p-th root of $b$ in an algebraic closure of $K$. Then, there is a unique derivation on $\d:K(c)\to K(c)$ extending that on $K$ such that $\d(c)=0$.
\end{fact}

\begin{lemma}\label{reducebasis}
Let $A$ and $B$ be disjoint subsets of $K$. Assume that $A\cup B$ is a differential $p$-basis for $(K,\d)$. Then, there is a differential field extension $L$ of $K$ such that $A$ is a differential $p$-basis for $(L,\d)$.
\end{lemma}
\begin{proof}
We assume $B=\{b\}$ (the general case followed by a standard transfinite induction). Let $b_1$ be a $p$-th root of $b$. By Fact~\ref{extendder}, there is a unique extension of the derivation to $K_1=K(b_1)$ with $\d(b_1)=0$. One can readily check that $A\cup\{b_1\}$ is a differential $p$-basis for $(K_1,\d)$. Repeating this process, we set $b_{i+1}$ to be a $p$-th root of $b_i$ and extend the derivation to $K_i(b_{i+1})$ such that $\d(b_{i+1})=0$. Again, it follows that $A\cup \{b_{i+1}\}$ is a differential $p$-basis for $(K_{i+1},\d)$. 

Now let $L=\bigcup_{i}K_i$. Since $A$ is differentially $p$-independent in each $(K_i,\d)$, we get that $A$ is differentially $p$-independent in $L$. Furthermore, from the construction, we see that $A$ is a differentially $p$-spanning set of $(L,\d)$. In other words, $A$ is a differential $p$-basis for $(L,\d)$, as desired. 
\end{proof}

This lemma generalises the fact that any differential field can be extended to a differentially perfect field (see \cite[Theorem 4]{Wood73}). On the other hand, there is a natural way to increase the differential degree of imperfection by passing to a transcendental extension.

\begin{lemma}
Let $T$ be a family of indeterminates (namely, algebraically independent over $K$). Let $\d$ be the unique derivation on $K(T)$ that extends $\d$ on $K$ and $\d(t)=0$ for all $t\in T$. If $A$ is a differential $p$-basis for $(K,\d)$, then $A\cup T$ is a differential $p$-basis for $(K(T),\d)$.
\end{lemma}
\begin{proof}
It suffices to consider the case when $T=\{t\}$. Clearly $A\cup{t}$ is differentially $p$-independent in $(K(t),\d)$. It remains to show that it is also a differentially $p$-spanning set of $(K(t),\d)$. For this, it suffices to show that $C_{K(t)}=C_K(t)$. The containment $C_K(t)\subseteq C_{K(t)}$ is clear. Now assume $a\in C_{K(t)}\setminus\{0\}$. Then 
$$a=\frac{p}{q}\quad \text{ and }\quad \d(a)=0$$
where $p,q\in K[t]$ are nonzero. We may assume that $p$ and $q$ have no common factors and that $q$ is monic. From $\d(a)=0$ we get
$$\d(p)\, q =p\, \d(q).$$
Since $\d(t)=0$, one derives $\d(p)=p^\d$ where the latter is the polynomial in $K[t]$ obtained by applying $\d$ to the coefficients of $p$. Similarly, $\d(q)=q^\d$ and note that, since $q$ is monic, $\deg q^\d<\deg q$. The above display becomes
$$p^\d \, q=p\, q^\d.$$
Since $p$ and $q$ are coprime and $\deg q^\d<\deg q$, the only way this equality can occur is if $p^\d=0$ and $q^\d=0$. In other words, $p,q\in C_K[t]$, and so $a\in C_K(t)$, as claimed.
\end{proof}

We now note that given a differential field extension $(L,\d)/(K,\d)$ with $L/K$ separable, if $A$ a differential $p$-basis for $(K,\d)$, then $A$ is a differential $p$-basis for $(L,\d)$ if and only if $C_L=L^p\cdot C_K$. Thus, the condition on when the constants of a separable extension change \emph{as little as possible} becomes relevant for building extensions preserving a differential $p$-basis. The following lemma says that this is the case when the extension is generated by a constrained tuple (refer to Section~\ref{diffpreli} for properties of constrained tuples; in particular, Remark~\ref{better}). 
 
\begin{lemma}\label{preserveconstants}
Let $\alpha$ be a constrained tuple over $K$. If $A$ is a differential $p$-basis for $(K,\d)$, then it is also a differential $p$-basis for $(K\langle \alpha\rangle, \d)$.
\end{lemma}
\begin{proof}
Since $K\langle \alpha\rangle/K$ is separable (by definition of constrained tuple), by the above observations it suffices to note that $C_{K\langle \alpha\rangle}=K\langle \alpha\rangle^p\cdot C_K$. But this is precisely the content of Remark~\ref{better}.
\end{proof}

As a consequence, we obtain the expected result on separably algebraic extensions.

\begin{corollary}
If $L$ is a differential extension of $K$ with $L/K$ separably algebraic, then any differential $p$-basis of $(K,\d)$ is a differential $p$-basis for $(L,\d)$. (In particular, the separable closure of $K$ has the same differential degree of imperfection as $K$).
\end{corollary}
\begin{proof}
Again it suffices to show that $C_L=L^p\cdot C_K$. We may assume that $L$ is finitely generated, say by a tuple $\alpha$. But now $\alpha$ is constrained over $K$, and so we can apply Lemma~\ref{preserveconstants}.
\end{proof}

We conclude this section by pointing out that, just as in the algebraic case, separability for differential field extensions is equivalent to preserving differential $p$-independent sets.

\begin{proposition}\label{characsep}
Let $(L,\d)/(K,\d)$ be a differential field extension. Then, the following are equivalent
\begin{enumerate}
\item $L/K$ is separable
\item $C_K$ and $L^p$ are linearly disjoint over $K^p$
\item Every differentially $p$-independent set of $(K,\d)$ is also differentially $p$-independent for $(L,\d)$
\item There exists a differential $p$-basis of $(K,\d)$ which is differentially $p$-independent for $(L,\d)$
\end{enumerate}
\end{proposition}
\begin{proof}
The equivalence (1)$\Leftrightarrow$(2) is the content of a result of Kolchin in \cite[Proposition 1, \S II.2]{Kolbook}.

The equivalences (2)$\Leftrightarrow$(3)$\Leftrightarrow$(4) follow from Lemma~\ref{propE2} by taking $E=C_K$ (as the intermediate field of $K$ and $K^p$).
\end{proof}

\subsection{Constrained constructions}\label{constructions}
In this section, we provide a way to construct separably differentially closed field extensions that preserve differential $p$-bases. As a result, this yields that $\SDCFe$ is a model consistent extension of DF$_{p,\epsilon}$ (namely, every model of DF$_{p,\epsilon}$ can be extended to a model of $\SDCFe$).

\medskip

Let $(K,\d)$ be a differential field of characteristic $p>0$. We construct a separably differentially closed field $(L,\d)$ extending $(K,\d)$ such that $L/K$ is separable and $C_L=L^p\cdot C_K$ (as a result any differential $p$-basis of $K$ will be one for $L$). The construction is as follows: Fix an enumeration of the pairs $(f_i,g_i)_{i\in I}$ where $f_i,g_i\in K\{x\}$ are nonzero, $s_{f_i}\neq 0$, and $\ord f_i>\ord g_i$. We may assume that $f_0$ is irreducible in $K\{x\}$. We can then find a solution $a$ of the system $(f_0=0) \land (g_0\neq 0)$ such that $K\langle a\rangle/K$ is a separable extension. Indeed, since $s_{f_0}\neq 0$, by Theorem~\ref{sepideal} the ideal $P=[f_0]:s_{f_0}^\infty$ of $K\{x\}$ is a separable prime differential ideal and $a=x+P$, from the fraction field of $K\{x\}/P$, is such a solution. By Lemma~\ref{specialise}, there exists a differential specialisation $\alpha$ of $a$ (over $K$) with $\alpha$ constrained over $K$ with constraint $g$. Thus, $f(\alpha)=0$ and $g(\alpha)\neq 0$. By Remark~\ref{better}, $C_{K\langle \alpha\rangle}=K\langle \alpha\rangle^p\cdot C_K$. Set $K_1:=K\langle\alpha\rangle$. Then $K_1$ is separable over $K$ and $C_{K_1}=K_1^p\cdot C_{K}$. By iterating this process (taking unions for limit ordinals), one builds a differential field $(E_1,\d)$ extension of $K$ such that $E_1/K$ is separable, $C_{E_1}=E_1^p\cdot C_K$, and $E_1$ solves every system $(f_i=0) \land (g_i\neq 0)$ for $i\in I$. Now repeat this construction to build $(E_{i+1},\d)$ from $(E_i,\d)$. Then, setting 
$$L=\bigcup_{i=1}^{\infty}E_i$$
we get a separably differentially closed field $(L,\d)$ extension of $K$ with $L/K$ separable and $C_L=L^p\cdot C_K$, as claimed. Any construction performed in the above fashion will be called a \emph{constrained construction over $K$}.

\smallskip

To summarise, the above construction yields:

\begin{proposition}\label{modelconsistent}
Let $(K,\d)$ be a differential field with differential $p$-basis $A$. Then, there exists a differential extension $L\models \SDCF$ having $A$ as a differential $p$-basis as well. 
\end{proposition}

We now prove that constrained constructions can be embedded into any separably differentially closed extension.

\begin{theorem}
Let $(L,\d)$ be a constrained construction over $K$. Then, for any separably differentially closed extension $(F,\d)$ of $K$, there exists a differential embedding $\phi:L\hookrightarrow F$ over $K$.
\end{theorem}
\begin{proof}
We use the notation from the constrained construction of $L$ above. Recall that $K_1=K\langle\alpha\rangle$ where $\alpha$ is constrained over $K$. By Corollary~\ref{isotype}(2), there are $f,g\in K\{x\}$ with $\ord g<\ord f$ such that the differential isomorphism type of $\alpha$ over $K$ is determined by $f(x)=0\land g(x)\neq 0$. Since $(F,\d)$ is separably differentially closed, there exists $\beta\in F$ with $f(\beta)=0$ and $g(\beta)\neq 0$. This yields that $I_\d(\alpha/K)=I_\d(\beta/K)$, and thus there is a differential $K$-isomorphism $K_1\cong K\langle\beta\rangle$ mapping $\alpha\mapsto \beta$. This of course induces a differential $K$-embedding $K_1\hookrightarrow F$. One can readily check that this process can be iterated in the constrained construction of $L$ to yield the desired differential $K$-embedding $L\hookrightarrow F$.
\end{proof}

The following is an immediate consequence of the above theorem.

\begin{corollary}\label{primemodel4}
Let $(K_1,\d_1)$ and $(K_2,\d_2)$ be differential fields and $\sigma:K_1\to K_2$ a differential isomorphism. If $(L,\d_1)$ is a constrained construction over $K_1$ and $(F,\d_2)$ is a separably differentially closed field extension of $K_2$, then there is an extension of $\sigma$ to differential embedding $\sigma':L\hookrightarrow F$.
\end{corollary}

\smallskip

\subsection{Completeness of $\SDCF_{p,\epsilon}$ for finite $\epsilon$}\label{compfinite}

In this subsection we assume that $\epsilon\in \mathbb N_0$ (i.e., $\epsilon$ is finite). While the theory SDCF$_{p,0}=$DCF$_p$ is model-complete \cite[Theorem 7]{Wood73}, we observe that, for $\epsilon>0$, the theory $\SDCFe$ is \emph{not} model-complete. Indeed, let $(K,\d)\models \SDCFe$ have differential $p$-basis $\bar a=(a_1,a_2,\dots,a_\epsilon)$. Let $c$ be a p-th root of $a_\epsilon$. By Fact~\ref{extendder}, there is a derivation $\d$ on $K(c)$ extending that on $K$ with $\d(c)=0$. Then, $(a_1,a_2,\dots,c)$ is a differential $p$-basis for $(K(c),\d)$. By Proposition~\ref{modelconsistent}, there is a separably differentially closed field $(L,\d)$ extending $(K(c),\d)$ having differential $p$-basis $(a_1,a_2,\dots,c)$; in particular, $(L,\d)\models \SDCFe$. Thus, we have two models $(K,\d)\subset (L,\d)$ of $\SDCFe$ where $\exists x\, (x^p=a_\epsilon)$ holds in $L$ but not in $K$. Hence, $\SDCFe$ is not model-complete (in the language of differential fields $\Ld$). 

Similar to the algebraic case, we obtain a model-completeness result after expanding the language by a differential $p$-basis. Namely, we expand the language $\Ld$ by constant symbols $\bar a=(a_1,\dots,a_\epsilon)$, denoted $\La$, where in a model of DF$_{p,\epsilon}$ they are to be interpreted as a differential $p$-basis. When $\epsilon=0$ we set $\bar a$ to be an empty tuple. We denote by DF$_{p,\epsilon}^{\bar a}$ the $\La$-theory consisting of DF$_{p,\epsilon}$ together with axioms specifying that $\bar a$ is a differential $p$-basis, and similarly for $\SDCFa$. To avoid confusion, we will write $K\subseteq_{\La}L$ to mean that this is an extension of $\La$-structures.

Our next result says that an $\La$-extension of models of DF$_{p,\epsilon}^{\bar a}$ is a separable extension (as fields).

\begin{lemma}\label{sepext}\label{sepextst}
Let $K$ and $L$ be models of DF$_{p,\epsilon}^{\bar a}$ with $K\subseteq_{\La}L$. Then, $L$ is separable over $K$.
\end{lemma}
\begin{proof}
Since $\bar a$ is a common differential $p$-basis, separability follows from Proposition~\ref{characsep}.
\end{proof}

Now model-completeness follows.

\begin{proposition}\label{modelcompletea}
The theory $\SDCFa$ is model-complete.
\end{proposition}
\begin{proof}
Let $K$ and $L$ be models of $\SDCFa$ with $K\subset_{\La}L$. It is enough to show that $K$ is existentially closed in $L$ (in the language $\La$). By Lemma~\ref{sepext}, $L/K$ is separable. Thus, since $(K,\d)$ is separably differentially closed, we have that $K$ is existentially closed in $L$ in the language $\Ld$. Now simply note that any quantifier-free $\La(K)$-formula is a quantifier-free $\Ld(K)$-formula. The result follows.
\end{proof}

By Proposition~\ref{modelconsistent}, any model of DF$_{p,\epsilon}^{\bar a}$ extends to a model of $\SDCFa$. This, together with the model-completeness result, tells us that $\SDCFa$ is the model companion of DF$_{p,\epsilon}^{\bar a}$. In fact, we can go further and observe below that $\SDCFa$ is the model-completion of DF$_{p,\epsilon}^{\bar a}$. To do this, we make use of the following amalgamation result of Kolchin. 

\begin{theorem}\cite[Proposition 4, \S II.2]{Kolbook}\label{Kolamal}
Let $(F_1,\d_1)$ and $(F_2,\d_2)$ be differential field extensions of $(K,\d)$ with $F_i/K$ separable for $i=1,2$. Then, there exist a differential extension $(E,\d)/(K,\d)$, with $E/K$ separable, and differential $K$-homomorphisms $g_i:F_i\to E$, for $i=1,2$, such that $E$ is the compositum of $g_1(F_1)$ and $g_2(F_2)$.
\end{theorem}

\begin{corollary}\label{modelcompletion} \
\begin{enumerate}
\item DF$_{p,\epsilon}^{\bar a}$ has the amalgamation property.
\item $\SDCFa$ is the model-completion of DF$_{p,\epsilon}^{\bar a}$. 
\end{enumerate}
\end{corollary}
\begin{proof}
(1) Let $K,F_1,F_2$ be models of DF$_{p,\epsilon}^{\bar a}$ with $\La$-embeddings $f_1:K\to F_1$ and $f_2:K\to F_2$. By Lemma~\ref{sepext}, the extensions $F_1/K$ and $F_2/K$ are separable. By Theorem~\ref{Kolamal}, there is a differential extension $(E,\d)/(K,\d)$, with $E/K$ separable, and $\La$-embeddings $g_1:F_1\to E$ and $g_2:F_2\to E$ with $g_1\circ f_1=g_2\circ f_2$. Since $E/K$ is separable, the differential $p$-basis $\bar a$ of $(K,\d)$ is differentially $p$-independent in $(E,\d)$. Thus, there is $B$ such that $\bar a\cup B$ is a differential $p$-basis for $(E,\d)$. By Lemma~\ref{reducebasis}, there is an extension $(L,\d)/(E,\d)$ where $\bar a$ is a differential $p$-basis. Thus, $(L,\d)$ is the desired amalgam.

\medskip

Now (2) follows from (1) and Proposition~\ref{modelcompletea}, as a model-companion of a theory $T$ is a model-completion if and only if $T$ has amalgamation.
\end{proof}

We finish with the promised completeness result.

\begin{corollary}\label{completefinite1}
The theory $\SDCFe$ is complete.
\end{corollary}
\begin{proof}
It suffices to show that $\SDCFa$ is complete. By \cite[Theorem 4.2.3]{Robinson65}, the model-completion of a theory that has a model that embeds into every other model must be complete. In our case $\mathbb F_p(t_1,\dots,t_\epsilon)$ equipped with the trivial derivation is such a model of DF$_{p,\epsilon}^{\bar a}$. Indeed, for any model $(K,\d)$ of DF$_{p,\epsilon}^{\bar a}$, the differential $p$-basis $\bar a$ is algebraically independent over $\mathbb F_p$ (as $\mathbb F_p(\bar a)/\mathbb F_p$ is separable).
\end{proof}

\subsection{Completeness of $\SDCF_{p,\infty}$}\label{completeinfinitecase}
In this section, we expand the language of differential fields by natural predicates modelled by those in the algebraic case \cite[\S 1]{Wood73}. For each $n\in\mathbb N$, we let $Q_n$ be an $n$-ary relation symbol. We set $Q=(Q_n)_{n>0}$. We let DF$_{p,\infty}^{Q}$ be the $\mathcal L_{\d,Q}$-theory consisting of DF$_{p,\infty}$ together with axioms specifying that $Q_n(x_1,\dots, x_n)$ holds if and only if $(x_1,\dots,x_n)$ is differentially $p$-independent (i.e., each entry is a constant and their $p$-monomials are linearly independent over the field of $p$-th powers). Similarly for SDCF$_{p,\infty}^Q$. As in the previous section, we will write $K\subseteq_{\mathcal L_{\d,Q}} L$ to mean that this is an extension of $\mathcal L_{\d,Q}$-structures. As in Lemma~\ref{sepextst}, when $K$ and $L$ are models of DF$_{p,\infty}^Q$, the extension $L/K$ is separable if and only if $K\subseteq_{\mathcal L_{\d,Q}} L$.

\begin{proposition}\label{modcompinfty} \
\begin{enumerate}
\item DF$_{p,\infty}^Q$ has the amalgamation property.
\item SDCF$_{p,\infty}^Q$ is the model-completion of DF$_{p,\infty}^Q$.
\end{enumerate}
\end{proposition}
\begin{proof}
(1) The proof is an easy adaptation of Corollary~\ref{modelcompletion}(1).

\smallskip

(2) Note that by Proposition~\ref{modelconsistent} any model of DF$_{p,\infty}^Q$ extends to a model of SDCF$_{p,\infty}^Q$. Thus, by part (1), it suffices to show that SDCF$_{p,\infty}^Q$ is model-complete. Let $K$ and $L$ be models of SDCF$_{p,\infty}^Q$ with $K\subseteq_{\mathcal L_{\d,Q}} L$. It is enough to show that $K$ is existentially closed in $L$ (in the language $\mathcal L_{\d,Q}$). Let $\phi(\bar x)$ be a quantifier-free $\mathcal L_{\d,Q}(K)$-formula and assume there is $\bar b$ from $L$ such that $L\models \phi(\bar b)$. Let $A$ be a differential $p$-basis for $(K,\d)$. We now consider two cases.

\medskip

\noindent \underline{Case 1. $A$ is a differential $p$-basis for $L$.} In this case we prove that there exists an existential $\Ld(K)$-formula $\rho(\bar x)$ with $L\models \rho(\bar b)$ and
$$L\models \forall \bar x \; (\rho(\bar x)\rightarrow \phi(\bar x)).$$
Note that this is enough to finish the proof in this case (as $K$ is existentially closed in $L$ in the language $\Ld$). Note that, since $\phi(\bar x)$ is a quantifier-free $\mathcal L_{\d,Q}(K)$-formula, we may assume that $\phi$ is a conjunction of equations, inequations, and formulas $Q$ and $\neg Q$. We construct $\rho$ as follows. Replace each appearance of $\neg Q_n(\hat x)$, where $\hat x=(x_1,\dots,x_n)$ is a sub-tuple of $\bar x$, with 
$$(\lor_{i=1}^n\d x_i\neq 0)\; \lor \; (y_{0}^pm_0(\hat x)+\cdots +y_s^p m_s(\hat x)=0 \text{ with not all } y_i\text{'s zero} )$$
and add $\exists y_0,\dots,y_s$ in front of the new formula (recall that the $m_i$'s denote the $p$-monomials). After doing this, we obtain $\rho'$ such that $L\models \forall x (\phi \leftrightarrow \rho')$. Now, for each appearance of $Q_n(\hat x)$, since $L\models Q_n(\hat b)$, we have that the tuple $\hat b=(b_1,\dots,b_n)$ is differentially $p$-independent. Since $A\subseteq K$ is a differential $p$-basis for $(L,\d)$, we can find distinct elements $a_1,\dots,a_m$ from $A$ and $\beta_{n+1},\dots,\beta_{m}$ also from $A$ such that 
$$L^p(a_1,\dots,a_m)=L^p(b_1,\dots,b_n,\beta_{n+1},\dots,\beta_m).$$
Then, the existential $\Ld(K)$-formula $\psi(\hat x)$ given by
\begin{align*}
\exists_{j} z_{j}\,\exists_{i,k} w_{i,k}\; \land_{i=1}^n\d x_i= 0 \; \land & \; \land_{i=1}^m (a_i=w_{i,1}^px_1+\cdots+w_{i,n}^px_n+w_{i,n+1}^pz_{n+1}+\cdots+w_{i,m}^pz_m) \\
\land & \; \land_{j=n+1}^m \delta z_j=0
\end{align*}
 is satisfied by $\hat b$ and clearly $L\models \forall \hat x(\psi\rightarrow Q_n)$. Replacing all appearances of $Q_n$ by the corresponding $\psi$, we obtain the desired $\rho$.

\bigskip

\noindent \underline{Case 2. $A$ is {\bf not} a differential $p$-basis for $L$.} Let $B$ be a differential $p$-basis for $L$ with $A\subset B$. Let $C=B\setminus A$. Note that $C$ is algebraically independent over $K$ and that the field $K(C)$ is a differential subfield of $L$. Furthermore, $K(C)$ has $B$ as a differential $p$-basis. Let $(F,\d)$ be a constrained construction in $L$ over $K(C)$ in the sense of Section~\ref{constructions}. Then $F\models \SDCF$ and has $B$ as differential $p$-basis. By \underline {Case 1}, we obtain that $F$ is existentially closed in $L$ (in the language $\mathcal L_{\d,Q}$), and hence $F$ has a realisation of $\phi$. 

Add a new set of constant symbols $\mathcal C$ to the language $\mathcal L_{\d,Q}$ such that $|\mathcal C|=|C|$. We now claim that the sentence $\exists\bar x\; \phi(\bar x)$ is implied by the following
\begin{align*}
\Sigma=\text{Diag}(K)\; \cup \; \SDCF_{p,\infty}^Q\; &\cup\;  \{Q_n(\bar c): \text{whenever }\bar c\in \mathcal C^n \text{ has distinct entries}\} \\
&\cup \{\mathcal C \text{ is algebraically independent over }K\}.
\end{align*}
Let $E$ be a model of $\Sigma$. Then $K$ is a differential subfield of $E$, $\mathcal C^E$ is differentially $p$-independent in $E$, and $\mathcal C^E$ algebraically independent over $K$. Thus, the obvious map $\sigma:K(C)\to K(\mathcal C^E)$ is a differential $K$-isomorphism. By Corollary~\ref{primemodel4}, there is a differential $K$-embedding $\sigma':F\to E$. The image of $F$ under $\sigma'$ is an $\mathcal L_{\d,Q}$-substructure of $E$ (since $\mathcal C^E$ will be a differential $p$-basis for $\sigma'(F)$ and is also differentially $p$-independent in $E$), and hence the image of the realisation of $\phi$ in $F$ will be a realisation of $\phi$ in $E$. This shows that $\exists \bar x \phi(\bar x)$ is indeed implied by $\Sigma$.

By compactness, there is a single $Q_n$, a single (algebraic) polynomial $f\in K[\hat t]$, and a tuple $\hat c$ from $\mathcal C$ (with $\hat t$ and $\hat c$ of the same length) such that 
$$\Sigma_0(\hat c)=\text{Diag}(K)\; \cup \; \SDCF_{p,\infty}^Q\; \cup\{Q_n(\hat c)\}\cup \{f(\hat c)\neq 0\}$$
implies $\exists \bar x\phi(\bar x)$. To complete the proof, it suffices to show that $K\models \exists \hat x \Sigma_0(\hat x)$. 
To prove this, take distinct $a_1,\dots,a_n$ from $A$ (the differential $p$-basis of $K$), and consider the inequation
$$f(\lambda_1 a_1,\dots,\lambda_n a_n)\neq 0$$
where the $\lambda_i$'s are varying in $K^p$. Since $K\models \SDCF$, $K^p$ is infinite and hence we can find nonzero values in $K^p$ for the $\lambda_i$'s such that the inequation is satisfied. Then, setting $\bar d=(\lambda_1a_1,\dots,\lambda_na_n)$ yields the desired tuple. This completes the proof.
\end{proof}

\begin{corollary}
The theory $\SDCF_{p,\infty}$ is complete.
\end{corollary}
\begin{proof}
It suffices to show that SDCF$_{p,\infty}^Q$ is complete. As in the proof of Corollary~\ref{completefinite1}, it suffices to show that there is a model of DF$_{p,\infty}^Q$ that embeds into any other model. In this case the field $\mathbb F_p(t_1,t_2,\dots)$ equipped with the trivial derivation is such a model of DF$_{p,\infty}^Q$.
\end{proof}

\section{Model theoretic properties of $\SDCFl$}\label{finalprop}

In this section, we establish further model-theoretic properties of the theory $\SDCFe$ (for $\epsilon\in \mathbb N_0\cup\{\infty\}$). We expand the language by the differential analogue of the algebraic $\lambda$-functions and denote this theory by $\SDCFl$. We prove that $\SDCFl$ has quantifier elimination, is a stable (but not superstable) theory and prime model extensions exist and are unique (up to isomorphism).

We carry on the notation from previous sections. In particular, $(K,\d)$ is a field of characteristic $p>0$ of differential degree of imperfection $\epsilon$.

\subsection{Quantifier elimination for finite $\epsilon$}\label{qefinite} In this subsection $\epsilon$ is finite. Note that the theory $\SDCFe$ does not admit quantifier elimination. For instance, for a model $(K,\d)$, the subfield $K^{p^2}$ is not quantifier-free definable, being a proper infinite subfield of $C_K$. As in the algebraic setup, we prove a quantifier elimination result by first expanding the language to include what we call the \emph{differential $\lambda$-functions}.

\medskip

Let $\bar a=(a_1,\dots,a_\epsilon)$ be a differential $p$-basis for $(K,\d)$, set $s=p^\epsilon -1$, and fix an order of the set of $p$-monomials $m(\bar a)=\{m_0(\bar a), \dots, m_s(\bar a)\}$. The differential $\lambda$-functions of $(K,\d)$ are defined as the functions $\ell_i:K\to K$, for $i=0,\dots,s$, determined by the conditions
$$
(\dagger)\quad \left\{
\begin{matrix}
\; \ell_i(b)=0 & \text{ if }b\notin C_K \\
\; b= \ell_0(b)^pm_0(\bar a)+\cdots+\ell_s(b)^pm_s(\bar a)& \text{if } b\in C_K
\end{matrix}
\right.
$$
for all $b\in K$.

\begin{remark}
Note that the differential $\lambda$-functions do not generally commute with the derivation. For example, letting $\epsilon=0$, we get that $\ell_0:K\to K$ is the $p$-th root function on $C_K$ and zero elsewhere. If $\d$ is nontrivial on $K$, then for any $a\notin C_K$ we have
$$\ell_0(\d(a^p))=0 \quad \text{ while }\quad \d(\ell_0(a^p))=\d(a)\neq 0.$$
\end{remark}

\

We now expand the language $\Ld$ of differential fields by constant symbols $\bar a=(a_1,\dots,a_\epsilon)$ and unary-function symbols $\ell_0,\dots,\ell_s$. We denote this by $\Ll$. In a model of DF$_{p,\epsilon}$ these new symbols are to be interpreted as a differential $p$-basis and the differential $\lambda$-functions, respectively. Note that when $\epsilon=0$, the tuple $\bar a$ is the empty tuple and there is a single differential $\lambda$-function $\ell_0$ (which is the $p$-th root function on constants, cf. Section~\ref{modelpreli}). We denote by DF$_{p,\epsilon}^{\ell}$ the $\Ll$-theory consisting of DF$_{p,\epsilon}$ together with axioms specifying that $\bar a$ is a differentially $p$-independent set and the $\ell_i$'s satisfy condition $(\dagger)$ above (it then follows that $\bar a$ is a differential $p$-basis, but this is not part of the axioms). Similarly, for $\SDCFl$. We note that DF$_{p,\epsilon}^{\ell}$ is a universal theory.

As in the previous section, to avoid confusion, we will write $K\subseteq_{\Ll}L$ to denote an extension of $\Ll$-structures. The results in Section~\ref{compfinite}, see Lemma~\ref{sepext} and Corollary~\ref{modelcompletion}(1), yield the following.

\begin{lemma}\label{adapt} \
\begin{enumerate}
\item 
Let $(L,\d)\models\DF_{p,\epsilon}^{\ell}$ and $K\subseteq_{\Ld} L$. Then, $K \subseteq_{\Ll} L$ if and only if $L/K$ is separable. In particular, if $K\subseteq_{\Ll} L$ then $K\models\DF_{p,\epsilon}^{\ell}$.
\item Every model of $\DF_{p,\epsilon}^{\ell}$ can be extended to a model of $\SDCFl$
\item $\DF_{p,\epsilon}^{\ell}$ has the amalgamation property.
\end{enumerate}
\end{lemma}

We can now prove quantifier elimination.

\begin{theorem}
The theory $\SDCFl$ has quantifier elimination.
\end{theorem}
\begin{proof}
We claim that $\SDCFl$ is model-complete. Let $(K,\d)$ and $(L,\d)$ be models of $\SDCFl$ with $K\subseteq_{\Ll}L$. Then, $K$ and $L$ are in particular $\La$-structures that are models of $\SDCFa$ with $K\subseteq_{\La}L$ (cf. Section~\ref{compfinite}). By model-completeness of $\SDCFa$, we get that $K$ is an elementary $\La$-substructure of $L$. Since the differential $\lambda$-functions are $\La$-definable, we get that $K$ is also an elementary $\Ll$-substructure of $L$. Thus, $\SDCFl$ is model-complete.

Now, by Lemma~\ref{adapt}, $\SDCFl$ is the model-completion of DF$_{p,\epsilon}^{\ell}$. Since the latter is a universal theory, the model completion admits quantifier elimination.
\end{proof}

\begin{remark}\label{qedcfp}
As a special case, we obtain that SDCF$_{p,0}^{\ell}$ admits quantifier elimination. Recall that in this case there is a single differential $\lambda$-function $\ell_0$, which is the $p$-th root function on constants. In other words, we recover Wood's result that DCF$_p^r$ admits q.e. \cite[Corollary 2.6]{Wood76}. 
\end{remark}

\subsection{Quantifier elimination for $\epsilon=\infty$}\label{qeinfinite}
Let $A$ be a differential $p$-basis for $(K,\d)\models$DF$_{p,\infty}$. For each $n\in \NN$, let $s_n=p^n-1$ and fix an enumeration of the $p$-monomials (as functions) $\{m_0,\dots,m_{s_n}\}$. In the case of the infinite differential degree of imperfection, we define the differential $\lambda$-functions $\ell_{n,i}:K^{n+1}\to K$, for $n\in \NN$ and $i=0,\dots,s_n$, as follows: let $(\bar a; b)=(a_1,\dots,a_n;b)\in K^{n+1}$, if an entry of $(a_1,\dots,a_n;b)$ is not in $C_K$ then $\ell_{n,i}(a_1,\dots,a_n;b)=0$; assuming now that all entries are in $C_K$, if $(a_1,\dots,a_n)$ are differentially $p$-dependent or $(a_1,\dots,a_n;b)$ are differentially $p$-independent, then $\ell_{n,i}(a_1,\dots,a_n;b)=0$ for $i=0,\dots,s_n$; otherwise, they are (uniquely) determined by
$$(\#)\quad \quad b=\ell_{n,0}(a_1,\dots,a_n;b)^pm_0(\bar a)+\cdots+\ell_{n,s_n}(a_1,\dots,a_n;b)^pm_{s_n}(\bar a).$$


We expand the language $\Ld$ by function symbols $\ell_{n,i}$ with $n\in \NN$ and $i=0,\dots,s_n$ (where $\ell_{n,i}$ has arity $n+1$). We denote this new language by $\Ll$. We denote by DF$_{p,\infty}^\ell$ the $\Ll$-theory consisting of DF$_{p,\infty}$ together with axioms specifying that the $\ell_{n,i}$'s satisfy the conditions above. Similarly for SDCF$_{p,\infty}^\ell$. As before, to avoid confusion, we will write $K\subseteq_{\Ll}L$ to denote an extension of $\Ll$-structures.

\smallskip

We leave the proof of the following to the reader.

\begin{lemma}\label{subaremodels}
Let $L\models$DF$_{p,\infty}^\ell$ and $K\subseteq_{\Ld}L$. Then, $K\subseteq_{\Ll}L$ if and only if $L/K$ is separable. 
\end{lemma}



We now prove the quantifier elimination result. 

\begin{theorem}
The theory $\SDCF_{p,\infty}^\ell$ has quantifier elimination. In particular, $\SDCF_{p,\infty}^\ell$ is the model-completion of DF$_{p,\infty}^\ell$.
\end{theorem}
\begin{proof}
Let $K$ and $L$ be models of $\SDCF_{p,\infty}^\ell$ with $K$ countable and $L$ $\omega_1$-saturated. Also, let $E$ and $F$ be $\Ll$-substructures of $K$ and $L$, respectively, and $\sigma:E\to F$ an $\Ll$-isomorphism. By Schoenfield test, quantifier elimination will follow once we show that $\sigma$ can be extended to an elementary $\Ll$-embedding from $K$ to $L$. 

By Lemma~\ref{subaremodels}, $K/E$ and $L/F$ are separable. Let $A$ and $B$ be differential $p$-bases for $E$ and $F$, respectively. If $A$ is finite, let $C$ be countably infinite such that $A\cup C$ is differentially $p$-independent in $K$; if $A$ is already infinite, set $C$ to be empty. Then, $A\cup C$ is a differential $p$-basis for $F(C)$, $C$ is algebraically independent over $F$, and $K/F(C)$ is separable. Choose similarly $D$ in $L$ (such that $B\cup D$ is differentially $p$-independent in $L$); it then follows that $F(C)$ and $E(D)$ are $\Ll$-isomorphic and are models of DF$_{p,\infty}^\ell$. 

Let $M$ be a constrained construction over $F(C)$ inside $K$ (as in Section~\ref{constructions}). Then $M\models \SDCF$ and has $C$ as a differential $p$-basis. It follows from the latter that $K/M$ is separable and hence $M\subseteq_{\Ll} K$. By Corollary~\ref{primemodel4}, we can extend the $\Ll$-isomorphism between $F(C)$ and $E(D)$ to an $\Ld$-embedding $\sigma':M\to L$. Since $M\subseteq_{\Ll} K$ and the differential $\lambda$-functions are $\Ld$-definable, $\sigma'$ is an $\Ll$-embedding. Let $M'=\sigma'(M)$. Then, $M'\models \SDCF_{p,\infty}$ and is an $\Ll$-substructure of $L$ (so $L/M'$ is separable). We can naturally make $M'$ and $L$ into $\mathcal L_{\d,Q}$-structures (see Section~\ref{completeinfinitecase}) and both will be models of $\SDCF_{p,\infty}^Q$. Furthermore, since $L/M'$ is separable, we get $M'\subseteq_{\mathcal L_{\d,Q}}L$. By model-completeness of $\SDCF_{p,\infty}^Q$ (see Theorem~\ref{modcompinfty}), we get that $M'$ is an elementary $\Ld$-substructure of $L$. Using again that the differential $\lambda$-functions are $\Ld$-definable, we see that $M'$ is elementary $\Ll$-substructure of $L$; in other words, $\sigma':M\to L$ is an elementary $\Ll$-embedding. As $L$ is $\omega_1$-saturated and $\SDCF_{p,\infty}$ is complete, we can further extend $\sigma'$ to an elementary $\Ll$-embedding $K\hookrightarrow L$, as desired. 
\end{proof}

\subsection{Stability} Let $\epsilon\in \mathbb N_0\cup\{\infty\}$. In this section, we prove that $\SDCFe$ is a stable theory. This generalises the result that DCF$_p$ is stable \cite{Shelah73}. Our proof makes use of a result of Kolchin (Proposition~\ref{changects} below) and our strategy differs from that used in the proof for DCF$_p$ presented in \cite{Shelah73,Wood76}. It does resemble the proof of the stability of SCF$_{p,e}$ presented by Srour in \cite[Proposition 8]{Srour86}.

\medskip

In general, for a differential field extension $(L,\d)/(K,\d)$, the constants $C_L$ need not equal $L^p\cdot C_K$. Nonetheless, when $L$ is differentially finitely generated over $K$, the following result of Kolchin says that $C_L$ is not much larger than $L^p\cdot C_K$.

\begin{proposition}\cite[Corollary 1,\S II.11]{Kolbook}\label{changects}
Let $(L,\d)/(K,\d)$ be a differentially finitely generated extension. Then, there exists a finite tuple $\beta$ from $C_L$ such that $C_L=L^p\cdot C_K(\beta)$ (namely, $C_L$ is finitely generated as a field over $L^p\cdot C_L$).
\end{proposition}

Before we prove stability, we note that $\SDCFe$ is \emph{not} superstable. Indeed, if it were superstable, any model would be algebraically closed by the superstable version of Macintyre's theorem \cite{CherShe80}. However, models of $\SDCF_p$ have infinite (algebraic) degree of imperfection by Lemma~\ref{basicprop}.

\medskip

Below, for finite $\epsilon$, the theory $\SDCFl$ is the one introduced in Section \ref{qefinite}; while $\SDCF_{p,\infty}^\ell$ is the theory introduced in Section \ref{qeinfinite} (the difference being the differential $\lambda$-functions).

\begin{theorem}
The theory $\SDCFl$ is stable.
\end{theorem}
\begin{proof}
Let $K\subset_{\Ll} L$ be models of $\SDCFl$ and $a\in L$. Let $F$ be the $\Ll$-structure generated by $a$ over $K$. We claim that $F$ is (differentially) countably generated as a differential field extension of $K$. We treat the cases of finite and infinite degree separately.
\medskip

\noindent \underline{\bf Finite $\epsilon$.} Let $F_0=K\langle a\rangle$. By Proposition~\ref{changects}, there is a finite tuple $b_0$ such that 
$$C_{F_0}=F_0^p\cdot C_{K}(b_0).$$
Let $F_1=F_0\langle \ell(b_0)\rangle$, where $\ell(b_0)$ denotes the (finite) tuple obtained by applying each $\ell_i$ to each entry of $b_0$. Again, Proposition~\ref{changects} yields a tuple $b_1$ such that
$$C_{F_1}=F_1^p\cdot C_{F_0}(b_1).$$
Continue in this fashion and let $F_{i+1}=F_i\langle\ell(b_i)\rangle$, there is a tuple $b_{i+1}$ such that
$$C_{F_{i+1}}=F_{i+1}^p\cdot C_{F_i}(b_{i+1}).$$
It follows, by construction, that $\cup_i F_i$ is the smallest $\Ll$-structure containing $a$ and $K$, and hence $F=\cup F_i$. Indeed, by construction, $\cup F_i$ is a differential subfield of $F$ with differential $p$-basis $\bar a$ (recall that, when $\epsilon$ is finite, $\bar a$ is a differential $p$-basis for $F$ and $L$). Thus, by Proposition~\ref{characsep}, $L$ is separable over $\cup F_i$, and, by Lemma~\ref{adapt}(1), this yields that $\cup F_i$ is an $\Ll$-substructure of $L$.  It follows that $F$ is countably generated as a differential field over $K$.

\medskip

\noindent \underline{\bf Infinite $\epsilon$.} Let $A$ be a differential $p$-basis for $K$. Let $F_0=K\langle a\rangle$. By Proposition~\ref{changects}, there is a finite tuple $b_0$ such that 
$$C_{F_0}=F_0^p\cdot C_{K}(b_0).$$
We can write $b_0=(b_{0,0},b_{0,1})$ where $A\cup b_{0,0}$ is $p$-independent over $L^p$ and $b_{0,1}$ is a tuple from $L^p(a_1,\dots,a_m,b_{0,0})$ where $a_1,\dots,a_m$ are distinct elements from $A$. Let $\ell(b_{0,1})$ denote the finite set of elements of the form $\ell_{n,i}(a_1,\dots,a_m,b_{0,0};b)$ where $b$ varies in the entries of $b_{0,1}$ and $n=m+length(b_{0,0})$.

Let $F_1=F_0\langle \ell(b_{0,1})\rangle$. Again, Proposition~\ref{changects} yields a tuple $b_1$ such that
$$C_{F_1}=F_1^p\cdot C_{F_0}(b_1).$$
Continue in this fashion and let $F_{i+1}=F_i\langle\ell(b_i)\rangle$, there is a tuple $b_{i+1}$ such that
$$C_{F_{i+1}}=F_{i+1}^p\cdot C_{F_i}(b_{i+1}).$$
It follows, by construction, that $\cup_i F_i$ is the smallest $\Ll$-structure containing $a$ and $K$, and hence $F=\cup F_i$. Indeed, by construction, $\cup F_i$ is a differential subfield of $F$ with differential $p$-basis $A\cup(\cup_i b_{i,0})$. Thus, by Proposition~\ref{characsep}, $L$ is separable over $\cup F_i$, and, by Lemma~\ref{subaremodels}, this yields that $\cup F_i$ is an $\Ll$-substructure of $L$.  It follows that $F$ is countably generated as a differential field over $K$.

\medskip


Since the differential $\lambda$-functions are definable in the language $\Ld$ (for $\Ll$-structures), the $\Ll$-isomorphism type of $F$ is determined by the $\Ld$-isomorphism type. The latter is determined by the defining differential ideal (over $K$) of the countably-many generators of $F$ in a differential polynomial ring over $K$ in countably-many (differential) variables.

A differential polynomial ring over $K$ in countable-many (differential) variables is also a polynomial ring over $K$ in countably-many (algebraic) variables. By Hilbert's basis theorem, there are at most $|K|^{\aleph_0}$ ideals in such polynomial rings. Hence, there are at most $|K|^{\aleph_0}$-many $\Ll$-isomorphism types for $L$. By quantifier elimination of $\SDCFl$, there are at most $|K^{\aleph_0}|$-many 1-types over $K$. In other words, $\SDCFl$ is $\kappa$-stable for any $\kappa$ with $\kappa=\kappa^{\aleph_0}$.
\end{proof}

\subsection{Prime model extensions} In this section we note that the theory $\SDCFl$ has unique prime model extensions. 

\begin{theorem}
The theory $\SDCFl$ has unique (up to isomorphism) prime model extensions. Furthermore, if $L$ is a prime model over $K$, then $L/K$ is atomic and $C_L=L^p\cdot C_K$.
\end{theorem}
\begin{proof}
Since $\SDCFl$ has quantifier elimination any embedding of models will be elementary. Thus, a constrained construction (in the sense of Section~\ref{constructions}) over any $K\models\;$DF$_{p,\epsilon}^{\ell}$ will be a prime model over $K$. Also, recall that any constrained construction $L$ over $K$ has the property $C_L=L^p\cdot C_K$. Since $\SDCFl$ is countable and stable, prime model extensions are unique up to isomorphism (see \cite[Chap. 9]{TZ12}, for instance). Furthermore, since prime model extensions exists, the countability of the theory yields that isolated types in $S_1^{\SDCFl}(K)$ are dense (see \cite[Ex.5.3.2]{TZ12}, for instance), and thus atomic extensions exist. As a consequence, prime models are atomic.
\end{proof}

We conclude by noting that our ``constrained constructions" (from Section~\ref{constructions}) are in fact constructible extensions in the usual model-theoretic sense, see \cite[Def.5.3.1(2)]{TZ12}. This follows from the following result.

\begin{proposition}\label{consiso}
Let $L\models \SDCFl$ be sufficiently saturated and $K\subset_{\Ll}L$. Let $\alpha$ be a (finite) tuple from $L$. If $\alpha$ is constrained over $K$, then $tp(\alpha/K)$ is isolated in the type space $S_n^{\SDCFl}(K)$.
\end{proposition}
\begin{proof}
Let $g$ be a constraint for $\alpha$. By Theorem~\ref{diffbasis} (i.e., the Differential Basis Theorem), there exist $f_1,\dots,f_s\in K\{\bar x\}$ such that the defining differential ideal of $\alpha$ is equal to the radical differential ideal generated by $f_1,\dots,f_s$. 
Let $\phi(\bar x)$ be the $\Ll$-formula
$$f_1(\bar x)=0 \; \land\; \cdots \;\land \; f_s(\bar x)=0\; \land\; g(\bar x)\neq 0. $$
We claim that $tp(\alpha/K)$ is isolated by $\phi$. Let $p\in S_n^{\SDCFl}(K)$ with $\phi\in p$. By saturation of $L$, there is a tuple $\beta$ from $L$ with $p=tp(\beta/K)$. Since $L\models \phi(\beta)$, we get $f_1(\beta)=\cdots=f_s(\beta)=0$, and so $I_\d(\alpha/K)\subseteq I_\d(\beta/K)$. Furthermore, since $g$ is a constraint for $\alpha$ and the extension $K\langle \beta\rangle/K$ is separable (as $K\subset_{\Ll}L$), we obtain $I_\d(\alpha/K)=I_\d(\beta/K)$. Thus, $K\langle \alpha\rangle\cong_{\Ld} K\langle \beta\rangle$ over $K$. Since $\alpha$ is constrained over $K$ so is $\beta$. Thus, by Lemma~\ref{preserveconstants}, $K\langle\alpha\rangle$ and $K\langle\beta\rangle$ are both $\Ll$-extensions of $K$. Since the differential $\lambda$-functions are $\Ld$-definable in $\Ll$-structures, it follows that
$$K\langle\alpha\rangle \cong_{\Ll(K)} K\langle \beta\rangle.$$
By quantifier elimination of $\SDCFl$, this yields $tp(\alpha/K)=tp(\beta/K)$. This shows that $\phi$ indeed isolates $tp(\alpha/K)$.
 \end{proof}

\

\section*{Appendix: Differentially separable extensions}

In \cite[\S II.8]{Kolbook}, Kolchin explores the notion of \emph{differentially separable extensions}. These are the differential analogues of separably algebraic extensions. One might naturally ask whether separably differentially closed fields can be characterised as those differential fields which are existentially closed in every differentially separable extension (cf. Theorem~\ref{sevchar}(2)). Here we note that this is {\bf not} the case. We prove below that a differential field $(K,\d)$ is existentially closed in every differentially separable extension if and only if $(K,\d)\models\,$ DCF.

In this section, we make no assumption on the characteristic of $K$. Let us recall some definitions from \cite[\S II.8]{Kolbook}. Let $L/K$ be a field extension and $(a_i)_{i\in I}$ an arbitrary family from $L$. Recall that the family $(a_i)_{i\in I}$ is said to be separably dependent over $K$ if there there exists a polynomial $f\in K[(t_i)_{i\in I}]$ vanishing at $(a_i)_{i\in I}$ such that at least one of the partial derivatives $\frac{\partial f}{\partial t_i}$ does \emph{not} vanish at $(a_i)_{i\in I}$.

\begin{remark}\label{charsep1}
In \cite[\S 0.4]{Kolbook}, Kolchin observes that a field extension $L/K$ is separable if and only if every family from $L$ that is algebraically dependent over $K$ is also separably dependent over $K$.
\end{remark}

\begin{definition}\cite[\S II.7-8]{Kolbook} \
Let $(L,\d)/(K,\d)$ be an extension of differential fields. 
\begin{itemize}
\item [(i)] A family $(a_i)_{i\in I}$ of elements from $L$ is said to be differentially separably dependent over $K$ if the family of derivatives $(\d^ja_i)_{i\in I,j\geq 0}$ is separably dependent over $K$. In the special case when the family $(a_i)_{i\in I}$ consists of a single element $a$, we simply say that $a$ is differentially separable over $K$.
\item [(ii)] We say that $L$ is differentially separable over $K$ if every element of $L$ is differentially separable over $K$.
\end{itemize}
\end{definition}

For several properties of differentially separable extensions, that resemble the properties of separably algebraic extensions, see \cite[\S II.8]{Kolbook}. We note here that if $(L,\d)/(K,\d)$ is differentially algebraic and $L/K$ is separable, then $(L,\d)/(K,\d)$ is differentially separable. Indeed, if $a\in L$, then $(\d^j a)_{j\geq 0}$ is algebraically dependent over $K$, but then by Remark~\ref{charsep1} this family must also be separably dependent over $K$. It turns out, however, that the converse is not true; namely, there are differentially separable extensions that are not separable (as fields). For instance, take any nonseparable field extension $L/K$, if we equip both fields with the trivial derivation then the extension will be differentially separable.

The above suggests that differentially separable extensions are not suitable to characterise separably differentially closed fields, in the manner of Theorem~\ref{sevchar}(2). In fact, our final result asserts this.

\begin{proposition}
Let $(K,\d)$ be a differential field. Then, $(K,\d)$ is existentially closed in every differentially separable extension if and only if $(K,\d)\models\DCF$.
\end{proposition}
\begin{proof}
($\Leftarrow$) is clear, as differentially closed fields are existentially closed in every differential field extension (in particular those that are differentially separable).
\smallskip

($\Rightarrow$) As differentially algebraic extensions that are separable are also differentially separable, the assumption in this direction yields that $K\models\SDCF$, see Theorem~\ref{sevchar}(1)-(2). In zero characteristic we are done. So we now assume $K$ has characteristic $p>0$. To prove that $(K,\d)\models \DCF_p$, it suffices to show that $(K,\d)$ is differentially perfect; namely, $C_K=K^p$. Let $b\in C_K$ and $c$ a $p$-th root of $b$ (in some extension). By Fact~\ref{extendder}, there is a derivation on $K(c)$ extending that on $K$ with $\d(c)=0$. By the latter, the element $c$ is differentially separable over $K$. Then, by \cite[Corollary (b), \S II.8]{Kolbook}, the extension $(K(c),\d)/(K,\d)$ is differentially separable. Thus, our assumption yields that $c\in K$, showing that $(K,\d)$ is differentially perfect. 
\end{proof}


\end{document}